\documentclass{article}
\usepackage[utf8]{inputenc}

\usepackage{amsmath,amssymb,amsthm}
\usepackage{stmaryrd}
\usepackage[colorlinks=true]{hyperref}
\usepackage{dsfont}
\usepackage{enumitem}
\usepackage{xcolor}
\usepackage{authblk}
\usepackage{bbold} 
\usepackage[a4paper]{geometry} 
\usepackage{soul}

\newcommand{\R}{\mathbb{R}}
\newcommand{\E}{\mathbb{E}} 

\newcommand{\dd}{\mathrm{d}}

\newcommand{\ind}[1]{\mathds{1}_{\{#1\}}}

\newtheorem{theorem}{Theorem}[section]
\newtheorem{lemma}[theorem]{Lemma}
\newtheorem{remark}[theorem]{Remark}
 
\newtheorem{proposition}[theorem]{Proposition}
\newtheorem{corollary}[theorem]{Corollary}

\newtheorem{assumptionA}{Assumption}
\newtheorem{assumptionAglob}{Assumption}
\newtheorem{assumptionH}{Assumption} 
 
\newtheorem{assumptionO}{Assumption}

\title{Quasi-stationary distribution for kinetic SDEs with low regularity coefficients} 

\author[1]{Nicolas Champagnat\thanks{E-mail: nicolas.champagnat@inria.fr.}}
\author[2]{Tony Lelièvre\thanks{E-mail: tony.lelievre@enpc.fr}} 
\author[3]{Mouad Ramil\thanks{E-mail: ramil.mouad@gmail.com}}
\author[4]{Julien Reygner\thanks{E-mail: julien.reygner@enpc.fr}}
\author[5,6]{Denis Villemonais\thanks{E-mail: denis.villemonais@math.unistra.fr}}

\affil[1]{Universit\'e de Lorraine, CNRS, Inria, IECL, Nancy, France}
\affil[2]{CERMICS, \'Ecole des Ponts, IP Paris and Inria, Marne-La-Vall\'ee, France}
\affil[3]{Research Institute of Mathematics, Seoul National University, Seoul, Republic of Korea} 
\affil[4]{CERMICS, École des Ponts, IP Paris, Marne-La-Vall\'ee, France}
\affil[5]{Université de Strasbourg, IRMA, Strasbourg, France}
\affil[6]{Institut Universitaire de France}

\date{}

\begin{document}

\maketitle
\begin{abstract}
  We consider kinetic SDEs with low regularity coefficients in the setting recently introduced in~\cite{Chaudru2022}. For the solutions to such equations, we first prove a Harnack inequality. Using the abstract approach of~\cite{CV}, this inequality then allows us to prove, under a Lyapunov condition, the existence and uniqueness (in a suitable class of measures) of a quasi-stationary distribution in cylindrical domains of the phase space. We finally exhibit two settings in which the Lyapunov condition holds: general kinetic SDEs in domains which are bounded in position, and Langevin processes with a non-conservative force and a suitable growth condition on the force.  

 \medskip

\end{abstract} 
 
\section{Introduction}
In statistical physics, the evolution of a molecular system at some fixed temperature $T$ is often modeled by the Langevin process:
\begin{equation}\label{eq:Langevin_intro}
  \left\{
    \begin{aligned}
        &\mathrm{d}q_t=M^{-1} p_t \mathrm{d}t , \\
        &\mathrm{d}p_t=F(q_t) \mathrm{d}t -\gamma M^{-1}  p_t
        \mathrm{d}t +\sqrt{2\gamma \mathrm{k} T} \mathrm{d}B_t,
    \end{aligned}
\right.  
\end{equation}
where $q_t$ (resp. $p_t$) is the position (resp. momentum) vector of the particles at time $t\geq0$. If $N$ is the number of particles then $(q_t,p_t)\in\mathbb{R}^{2d}$ where $d=3N$. Additionally, $M\in\mathbb{R}^{d\times d}$ is a symmetric positive definite matrix (the mass matrix), $\gamma>0$ is the friction coefficient, $F:\mathbb{R}^d\to\mathbb{R}^d$ is the interaction force between particles, and $\mathrm{k}$ is the Boltzmann constant. Most often, the force $F$ is conservative i.e. there exists a potential function $V$ such that $F=-\nabla V$.

The sampling of long time trajectories of~\eqref{eq:Langevin_intro} is fundamental to understand and model the successive transitions between metastable states of a  molecular system. Such sampling procedures have numerous
applications in biology, chemistry and materials science~\cite{Tuck,LelSto16}. The main numerical challenge comes from the long time scales required to compute such transitions. In this regard, several methods have been developed. Some of them~\cite{lelievre2020mathematical} rely on the existence of a quasi-stationary distribution (QSD) for the metastable states of the system, which corresponds to the long-time equilibrium reached by the system when it is trapped in a metastable state. More precisely, for a given domain $D$ in $\R^{2d}$ and the associated exit time $\tau_\partial=\inf \{t>0 : (q_t,p_t) \not \in D\}$, a probability measure $\nu$ with support in $D$ is a QSD if and only if, for any positive $t$, for any measurable set $A \subset \R^{2d}$,
$${\mathbb P}_\nu ((q_t,p_t) \in A, t < \tau_\partial) = \nu(A) \, {\mathbb P}_\nu (t < \tau_\partial). $$
It is standard, see~\cite{Collet,le-bris-lelievre-luskin-perez-12}, that starting from the QSD, the exit time is exponentially distributed and independent of the exit point, and this can be used to sample the first exit event faster in wall-clock time using parallel in time algorithm. This motivates the two following questions: the first one concerns the existence of a unique QSD for the considered system and the second is related to the attainability of this QSD, namely the long-time convergence to this unique QSD for the process conditioned to remain in the domain.

In the recent literature, these questions have been successfully studied for processes of the form~\eqref{eq:Langevin_intro} (and variations thereof) in domains of the form
\begin{equation*}
    D=\mathcal{O}\times\mathbb{R}^d,
\end{equation*}
where $\mathcal{O}$ is a smooth open set of $\mathbb{R}^d$. We shall refer to such domains as \emph{cylindrical}. In particular,~\cite{lelievre2022quasi,guillin2020quasi} employ a spectral approach which relies on the interpretation of QSDs as the principal eigenvector of the infinitesimal generator associated with~\eqref{eq:Langevin_intro}, complemented with suitable boundary conditions on $\partial D$~\cite{LelRamRey}. A key technical point in this approach is to obtain compactness, or quasi-compactness, of the associated semigroup, see also~\cite{ferre2021more,benaim2021degenerate,guillin2023quasi,guillin2024generalized,CarGabMedMis}. On the other hand, a more probabilistic approach to the study of QSDs was developed in~\cite{CV}, it essentially rests on the combination of a Harnack inequality and a Lyapunov condition.

The initial motivation of this article is thus to apply the approach from~\cite{CV} to the Langevin process~\eqref{eq:Langevin_intro} in cylindrical domains. We shall actually work in the more general setting of kinetic SDEs, of the form
\begin{equation}\label{eq:Langevin}
  \left\{
    \begin{aligned}
        &\mathrm{d}q_t=p_t \, \mathrm{d}t , \\
        &\mathrm{d}p_t=F(q_t,p_t) \,\mathrm{d}t +\sigma(q_t,p_t) \, \mathrm{d}B_t,
    \end{aligned}
\right.
\end{equation}
in $\R^{2d}$. Indeed, up to a standard change of variables~\cite[Remark 3.37]{2010free}, which reduces the mass matrix $M$ to the identity of $\R^d$, the Langevin process~\eqref{eq:Langevin_intro} is a particular case of~\eqref{eq:Langevin}. Moreover, recent results from~\cite{Chaudru2022} will also allow us to address coefficients $F : \R^{2d} \to \R^d$, $\sigma : \R^{2d} \to \R^{d \times d}$  with low regularity. 

The precise definition of our setting, together with the statement of our main results, are presented in Section~\ref{ss:main-results}. The rest of the article is then organized as follows: the Harnack inequality is proved in Section~\ref{sec:proof Harnack}, the application of the results from~\cite{CV} is detailed in Section~\ref{sec:critere qsd}, and two examples of settings in which a Lyapunov function can be explicitly constructed are presented in Section~\ref{sec:example lyapunov}.

\section{Main results}\label{ss:main-results}

The first main result of this work, Theorem~\ref{thm:Harnack}, is a Harnack inequality for the solution to the kinetic SDE~\eqref{eq:Langevin} in (a local-in-position version of) the low-regularity setting of~\cite{Chaudru2022}. It is introduced in Section~\ref{sss:main-harnack}. The second main result, Theorem~\ref{thm:critere qsd}, establishes that in the setting of Theorem~\ref{thm:Harnack} and under a supplementary Lyapunov condition, the solution to the kinetic SDE~\eqref{eq:Langevin} admits a unique QSD in cylindrical domains. It is detailed in Section~\ref{sss:main-qsd}. Last, in Section~\ref{sss:main-examples}, we provide two examples of settings in which the Lyapunov condition can be established explicitly, which implies that the conclusion of Theorem~\ref{thm:critere qsd} holds. 

\subsection{Harnack inequality in the low regularity setting}\label{sss:main-harnack}

Let $d\geq1$. We use the generic notation $x\in\mathbb{R}^{2d}$ for pairs $(q,p) \in \R^{2d}$, and similarly the solution to~\eqref{eq:Langevin} is denoted by $(X_t)_{t \geq 0} = (q_t,p_t)_{t \geq 0}$.

We denote by $\vert\cdot\vert$ the Euclidean norm on $\R^d$, and use the notation $|x|=|q|+|p|$ for $x=(q,p) \in \R^{2d}$. We denote by $\Vert\cdot\Vert$ the operator norm on $\R^{d \times d}$, and by $\Vert g\Vert_\infty$ the supremum norm of a (possibly vector-valued) function $g$. 

\medskip
Our main assumptions on the SDE~\eqref{eq:Langevin} are the following.

\begin{assumptionA}[Drift coefficient]\label{ass:A1loc}
    The vector field $F:\R^{2d} \to \R^d$ is measurable and satisfies, for any bounded set $\mathcal{O}^\flat \subset \R^d$, 
    \begin{equation*}
        \sup_{x,x' \in \mathcal{O}^\flat\times \R^d, |x-x'| \leq 1} |F(x)-F(x')| < \infty.
    \end{equation*}
\end{assumptionA}

\begin{assumptionA}[Diffusion coefficient]\label{ass:A2loc}
    The matrix field $\sigma : \R^{2d} \to \R^{d \times d}$ is measurable, and for any bounded set $\mathcal{O}^\flat \subset \R^d$, 
    \begin{itemize}
        \item there exist $0 < c_1 \leq c_2$ such that for all $x\in\mathcal{O}^\flat\times \R^d$, $\xi\in\R^d$, 
        \begin{equation*}
            c_1\vert\xi\vert^2\leq\xi\cdot\sigma\sigma^T(x)\xi\leq c_2\vert\xi\vert^2;
        \end{equation*}
        \item there exist $\alpha\in(0,1)$ and $c_3>0$ such that for all $x = (q,p), x'=(q',p') \in \mathcal{O}^\flat\times \R^d$,
        \begin{equation*}
            \Vert \sigma(x)-\sigma(x')\Vert\leq c_3(\vert q-q'\vert^{\alpha/3}+\vert p-p'\vert^\alpha).
        \end{equation*}
    \end{itemize}
\end{assumptionA}

\begin{assumptionA}[Well-posedness of~\eqref{eq:Langevin}]\label{ass:sde}
    For any $x=(q,p) \in \R^{2d}$, there exists a unique weak solution $(X_t)_{t \geq 0} = (q_t,p_t)_{t \geq 0}$ to~\eqref{eq:Langevin} starting from $x$. 
    
    Moreover, denoting by $\mathbb{P}_x$ the probability measure under which $X_0=x$:
    \begin{enumerate}
        \item under $\mathbb{P}_x$, for any $t > 0$, the random variable $X_t$ has a density $p(t;x,\cdot)$ with respect to the Lebesgue measure on $\R^{2d}$;
        \item for any $t>0$ and almost every $y \in \R^{2d}$, the function $x \mapsto p(t;x,y)$ is continuous on $\R^{2d}$.
    \end{enumerate}
\end{assumptionA}

Notice that the continuity in $x$ of the transition density implies, by Scheffé's lemma, that the Markov process $(X_t)_{t \geq 0}$ has the strong Feller property.

\medskip
To present a particular case in which Assumption~\ref{ass:sde} holds, we introduce the following \emph{global-in-position} versions of Assumptions~\ref{ass:A1loc} and~\ref{ass:A2loc}.

\begin{assumptionAglob}[Drift coefficient]\label{ass:A1glob}
    The vector field $F:\R^{2d} \to \R^d$ is measurable and satisfies 
    \begin{equation*}
        c_0 := \sup_{x,x' \in \R^{2d}, |x-x'| \leq 1} |F(x)-F(x')| < \infty.
    \end{equation*}
\end{assumptionAglob}

\begin{remark}[Assumption~\ref{ass:A1glob} and linear growth]\label{rk:linear-growth}
    If $F$ satisfies Assumption~\ref{ass:A1glob}, then there exist $a, b \geq 0$ such that
    \begin{equation*}
        \forall x \in \R^{2d}, \qquad |F(x)| \leq a + b|x|.
    \end{equation*}
\end{remark}

\begin{assumptionAglob}[Diffusion coefficient]\label{ass:A2glob}
    The matrix field $\sigma : \R^{2d} \to \R^{d \times d}$ is measurable, and:
    \begin{itemize}
        \item there exist $0 < c_1 \leq c_2$ such that for all $x\in\R^{2d}$, $\xi\in\R^d$, 
        \begin{equation*}
            c_1\vert\xi\vert^2\leq\xi\cdot\sigma\sigma^T(x)\xi\leq c_2\vert\xi\vert^2;
        \end{equation*}
        \item there exist $\alpha\in(0,1)$ and $c_3>0$ such that for all $x = (q,p), x'=(q',p') \in \R^{2d}$,
        \begin{equation*}
            \Vert \sigma(x)-\sigma(x')\Vert\leq c_3(\vert q-q'\vert^{\alpha/3}+\vert p-p'\vert^\alpha).
        \end{equation*}
    \end{itemize}
\end{assumptionAglob}

The setting of Assumptions~\ref{ass:A1glob} and~\ref{ass:A2glob} is the one introduced by Chaudru de Raynal, Menozzi, Pesce and Zhang in~\cite{Chaudru2022}, where it is argued to be optimal, in terms of regularity, for the well-posedness of~\eqref{eq:Langevin} in the weak sense. In particular, the next statement, which will be used in the proofs of our main results, follows from~\cite[Theorem~1.1]{Chaudru2022}.

\begin{proposition}[Well-posedness of~\eqref{eq:Langevin} under global conditions]\label{prop:chaudru}
    Let Assumptions~\ref{ass:A1glob} and~\ref{ass:A2glob} hold. Then Assumption~\ref{ass:sde} holds. 
    
    Moreover, for any $t>0$, there exists $C(t)$ which only depends on $F$ and $\sigma$ through the quantities $c_0,c_1,c_2,c_3,\alpha$ from Assumptions~\ref{ass:A1glob} and~\ref{ass:A2glob}, as well as $a,b$ from Remark~\ref{rk:linear-growth}, such that for any $x \in \R^{2d}$ and $s \in (0,t]$, the density $p(s;x,\cdot)$ of $X_s$ under $\mathbb{P}_x$ satisfies
    \begin{equation}\label{eq:chaudru-gaussian-bound}
        \forall y \in \R^{2d}, \qquad p(s;x,y) \leq \frac{C(t)}{s^{2d}}.
    \end{equation}
\end{proposition}

We now let $\mathcal{O} \subset \R^d$ be a nonempty set which satisfies the following assumption.

\begin{assumptionO}\label{ass:O}
    The set $\mathcal{O}$ is open, connected, and its boundary $\partial \mathcal{O}$ has Lebesgue measure $0$.
\end{assumptionO} 

The condition that the Lebesgue measure of $\partial \mathcal{O}$ be $0$ may be understood as a minimal regularity assumption on the boundary of $\mathcal{O}$.

\medskip
Given $\mathcal{O}$, let us define the cylindrical domain $D:=\mathcal{O}\times\R^d$. Under Assumption~\ref{ass:sde}, let $\tau_\partial$ be the first exit time from $D$ of~\eqref{eq:Langevin}, i.e. 
$$\tau_\partial:=\inf\{t>0: X_t\notin D\} \in [0,\infty].$$
The first main result of this work is the following Harnack inequality. 

\begin{theorem}[Harnack inequality]\label{thm:Harnack}
Let Assumptions~\ref{ass:A1loc}, \ref{ass:A2loc}, \ref{ass:sde} and~\ref{ass:O} hold. For any measurable subset $A$ of~$D$, define the function $u_A : [0, \infty) \times D \to [0,1]$ by
\begin{equation}\label{eq:def u_A}
   \forall (t,x) \in [0, \infty) \times D, \qquad u_A(t,x):=\mathbb{P}_x(t < \tau_\partial, X_t\in A). 
\end{equation}

Then, for any compact set $K\subset D$, $\epsilon>0$ and $T>0$, there exists a constant $C_{K,\epsilon,T}>0$ such that for any measurable set $A\subset D$ and for any $t\geq \epsilon$,
\begin{equation}\label{eq:ineq harnack thm} 
    \sup_{x\in K}u_A(t,x)\leq C_{K,\epsilon,T}\inf_{x\in K}u_A(t+T,x).
\end{equation} 
\end{theorem} 

We will also make use of the following corollary of Theorem~\ref{thm:Harnack}.

\begin{corollary}[Case of subsets with nonempty interior]\label{cor:Harnack}
    In the setting of Theorem~\ref{thm:Harnack}, if the set $A$ has nonempty interior, then $u_A(t,x)>0$ for all $t>0$, $x\in D$. As a consequence, for any compact set $K \subset D$, for all $t>0$, $\inf_{x \in K} u_A(t,x) > 0$.
\end{corollary}

The proofs of Theorem~\ref{thm:Harnack} and Corollary~\ref{cor:Harnack} are detailed in Section~\ref{sec:proof Harnack}. They essentially follow the lines of the proof of~\cite[Theorem~2.15]{LelRamRey}, which is stated for the Langevin process~\eqref{eq:Langevin_intro}, with smooth coefficients, and in a domain $\mathcal{O}$ which is bounded. Therefore, the extension of the proof requires to take care of the more general form of the kinetic SDE~\eqref{eq:Langevin}, the possible low regularity of the coefficients $F$ and $\sigma$, and the possible unboundedness of the domain $\mathcal{O}$.

\subsection{Lyapunov condition and QSD}\label{sss:main-qsd}

Under the assumptions of Theorem~\ref{thm:Harnack}, we now address the existence and uniqueness of a QSD for the process $(X_t)_{t \geq 0}$ in the set $D$. Our approach is based on the general criteria obtained in~\cite{CV}. They essentially require to combine the Harnack inequality from Theorem~\ref{thm:Harnack} with the following Lyapunov condition for the infinitesimal generator of $(X_t)_{t \geq 0}$, defined by 
\begin{equation}\label{generateur Langevin}
    \mathcal{L}_{F,\sigma} = p\cdot\nabla_q +F\cdot\nabla_p +\frac{1}{2}\sigma\sigma^T:\nabla^2_p,
\end{equation}
where $:$ is the Frobenius product applied to matrices.

\begin{assumptionH}[Lyapunov function]\label{ass:lyapunov}
For all $\lambda>0$, there exist a constant $c_\lambda>0$, a function $\phi_\lambda:D\to [1,\infty)$ which is $\mathcal{C}^1$ in $q$ and $\mathcal{C}^2$ in $p$, and a bounded subset $D_\lambda\subset D$ closed in $D$ such that 
$$\sup_{x\in D_\lambda}\phi_\lambda(x)<\infty,$$
and
$$\forall x \in D, \qquad \mathcal{L}_{F,\sigma}\phi_\lambda(x)\leq -\lambda\phi_\lambda(x)+c_\lambda\mathbb{1}_{D_\lambda}(x).$$ 
\end{assumptionH} 

Using the Harnack inequality from Theorem~\ref{thm:Harnack} we are able to prove that under Assumption~\ref{ass:lyapunov}, the criteria from~\cite[Theorem 3.5]{CV} are satisfied for the process~\eqref{eq:Langevin}. These criteria ensure the existence of a unique QSD in the family
\begin{equation*}
\mathcal{M}_{\lambda_0} = \{\text{$\mu$ probability measure on $D$}: \exists \lambda > \lambda_0, \exists p \in (1, \lambda/\lambda_0), \mu(\phi_\lambda^{1/p}) < \infty\}
\end{equation*}
for some $\lambda_0 \geq 0$, as well as the long-time convergence of the law of the process conditioned to remain in $D$ towards the QSD (for initial conditions in $\mathcal{M}_{\lambda_0}$), as stated in the next theorem, which is our second main result.

\begin{theorem}[QSD]\label{thm:critere qsd}
Let Assumptions~\ref{ass:A1loc}, \ref{ass:A2loc}, \ref{ass:sde}, \ref{ass:O} and~\ref{ass:lyapunov} hold. There exists $\lambda_0 \geq 0$ such that $(X_t)_{t \geq 0}$ has a unique QSD $\mu$ in $D$ in the family $\mathcal{M}_{\lambda_0}$. This measure actually satisfies $\mu(\phi_{\lambda}^{1/p})<\infty$ for all $\lambda > \lambda_0$ and $p \in (1, \lambda/\lambda_0)$.

Moreover, for any $\lambda > \lambda_0$ and $p \in (1, \lambda/\lambda_0)$, there exist  $C,\beta>0$, such that, for all $x\in D$ and all measurable $f:D\to\mathbb R$ with $|f| \leq \phi_\lambda^{1/p}$,
 \begin{equation}
 \label{thm-specgap}
 \left|\mathrm e^{\lambda_0 t}\mathbb E_x\left[\ind{t < \tau_\partial} f(X_t)\right]-\varphi(x)\mu(f)\right|\leq C\mathrm e^{-\beta t}\phi_\lambda^{1/p}(x),
 \end{equation}
where $\varphi:D\to\mathbb{R}^*_+$ is a positive function which does not depend on~$\lambda,p$.
\end{theorem}  

Theorem~\ref{thm:critere qsd} is proved in Section~\ref{sec:critere qsd}.

\begin{remark}
     We emphasize that~\eqref{thm-specgap} implies that for any $\lambda>\lambda_0$ and $p \in (1, \lambda/\lambda_0)$, there exists $C'>0$ such that, for any probability measure $\nu$ on $D$,
    \begin{equation}
    \label{eq:thm-QSD}
    \forall t\geq0,\qquad\Vert\mathbb{P}_\nu(X_t\in\cdot|\tau_{\partial}>t)-\mu(\cdot)\Vert_{TV}\leq C'\mathrm{e}^{-\beta t}\frac{\nu(\phi_\lambda^{1/p})}{\nu(\varphi)}.
    \end{equation}
    This follows for example from the arguments of~\cite[Thm.\ 2.2]{CV-ECP-2020}.
\end{remark}

\begin{remark}
When $\mathcal O=\mathbb R^d$, $\tau_\partial = \infty$ almost surely. In this situation, $\lambda_0=0$ and Theorem~\ref{thm:critere qsd} provides the existence of a stationary distribution, which is unique in $\mathcal{M}_0$. Similar results in the conservative case were obtained in~\cite{SachsLeimkuhlerEtAl2017}.
\end{remark} 

\begin{remark}
    Note that it is not necessary to assume that~\ref{ass:lyapunov} holds true for any $\lambda>0$. It is only needed for some $\lambda>\lambda_0$. We kept the simpler version~\ref{ass:lyapunov} because it is hard in practice to compute bounds for $\lambda_0$ and it simplifies the analysis.
\end{remark}

\subsection{Examples}\label{sss:main-examples}

We now provide examples of situations in which the assumptions of Theorem~\ref{thm:critere qsd} are satisfied. The first one is the case of bounded-in-position domains.

\begin{theorem}[QSD in a bounded domain]\label{thm:ex-bounded}
    Let $F$ and $\sigma$ satisfy Assumptions~\ref{ass:A1loc}, \ref{ass:A2loc} and~\ref{ass:sde}, and let $\mathcal{O}$ be a \emph{bounded} domain satisfying Assumption~\ref{ass:O}. Then Assumption~\ref{ass:lyapunov} is satisfied with a \emph{bounded} Lyapunov function $\phi_\lambda$, so that $(X_t)_{t \geq 0}$ has a unique QSD $\mu$ on $D$, and~\eqref{eq:thm-QSD} holds.
\end{theorem}

This statement generalises the result of~\cite[Theorems~2.14 and~2.22]{lelievre2022quasi}, where the existence and uniqueness of the QSD in a (smooth) bounded-in-position domain is proved for the Langevin process~\eqref{eq:Langevin_intro}, with smooth coefficients. The method of proof is however drastically different, as the argument of~\cite{lelievre2022quasi} is based on the compactness of the (absorbed) semigroup of $(X_t)_{t \geq 0}$, which follows from a quantitative Gaussian upper bound derived on the transition density of this process. The main technical contribution of the proof of Theorem~\ref{thm:ex-bounded} is thus the derivation of a simple bounded Lyapunov function. It is detailed in Section~\ref{ss:ex-bounded}.

\medskip
When the domain $\mathcal{O}$ is no longer assumed to be bounded, the construction of a Lyapunov function becomes more dependent on the particular expression of $F$ and $\sigma$. In~\cite{guillin2020quasi}, Guillin, Nectoux and Wu address the case of a force field $F$ which takes the form
\begin{equation*}
    F(q,p) = -\nabla V(q) - \gamma(q,p) p,
\end{equation*}
for a potential function $V : \R^d \to \R$, a nonconstant friction matrix $\gamma(q,p) \in \R^{d \times d}$, and a diffusion matrix $\sigma$ which writes
\begin{equation*}
    \sigma(q,p) = \Sigma(q,p) I_d,
\end{equation*}
for a non-degenerate and bounded scalar-valued diffusion coefficient $\Sigma(q,p)$. Under suitable regularity and growth conditions on $V$, $\gamma$ and $\Sigma$, they construct a Lyapunov function $\mathsf{W}_1(q,p)$, which is essentially a perturbation of $\exp(a H(q,p))$, for some $a>0$ and where $H(q,p)=V(q)+\frac{1}{2}|p|^2$ is the Hamiltonian of the system. This construction provides, for any $p > 1$, the existence and uniqueness of a QSD in the class of probability measures $\nu$ such that $\nu(\mathsf{W}_1^{1/p}) < \infty$. It can be checked that in this setting, the function $\mathsf{W}_1$ indeed satisfies our Assumption~\ref{ass:lyapunov}, so Theorem~\ref{thm:critere qsd} also applies to recover similar results.

\medskip
In Section~\ref{ss:ex-langevin}, we consider a different setting, which is studied in~\cite{LRS}. We come back to the case of the Langevin process~\eqref{eq:Langevin_intro}, with mass matrix $M=I_d$ for simplicity, and an interaction force which decomposes as
\begin{equation}\label{eq:ex-langevin:1}
    F(q) = -\nabla U(q) - \ell(q),
\end{equation}
where $U\in\mathcal{C}^1(\R^d)$, $U\geq0$, and there exist constants $\alpha>0$ and $\beta\in(0,\gamma)$ such that for all $q\in\mathbb{R}^d$,
\begin{equation}\label{eq:ex-langevin:2}
    \left(\nabla U(q)+\ell(q)\right) \cdot q \geq\alpha\left(|q|^{2}+U(q)\right)+\frac{|\ell(q)|^{2}}{\beta^{2}}.
\end{equation}
We only assume that the vector field $\ell$ is measurable, and notice that then Assumption~\ref{ass:A1loc} holds as soon as $\ell$ is bounded on bounded sets of $\R^d$. On the other hand, since we are dealing with the Langevin process~\eqref{eq:Langevin_intro}, for which the diffusion matrix $\sigma = \sqrt{2\gamma \mathrm{k} T} I_d$ is constant, Assumption~\ref{ass:A2loc} obviously holds true. In this setting, we obtain the following result.

\begin{theorem}[QSD for the Langevin process]\label{thm:ex-langevin}
    Consider the Langevin process~\eqref{eq:Langevin_intro}, with a mass matrix $M=I_d$ and an interaction force which satisfies~\eqref{eq:ex-langevin:1}--\eqref{eq:ex-langevin:2} for a measurable vector field $\ell : \R^d \to \R^d$ which is bounded on bounded sets of $\R^d$. If Assumption~\ref{ass:sde} holds true, then for any set $\mathcal{O} \subset \R^d$ which satisfies Assumption~\ref{ass:O}, the process $(X_t)_{t \geq 0}$ possesses a QSD $\mu$ in $D = \mathcal{O} \times \R^d$, and~\eqref{eq:thm-QSD} holds with the function $\phi_\lambda$ defined in Section~\ref{ss:ex-langevin}. Moreover, there exists $n \geq 0$ such that $\mu$ is the unique QSD in the class of probability measures $\nu$ such that $\nu((U(q)+|q|^2+|p|^2)^n) < \infty$.
\end{theorem}

The Lyapunov function employed to show that Assumption~\ref{ass:lyapunov} is satisfied is derived from results in~\cite{LRS}. It is a polynomial of a modification of the Hamiltonian $U(q)+\frac{1}{2}|p|^2$ of the system. As a consequence, it yields uniqueness of a QSD in a class of probability measures which requires less integrability than classes obtained with Lyapunov functions which are exponential of (modifications of) the Hamiltonian.

\section{Proof of Theorem~\ref{thm:Harnack} and Corollary~\ref{cor:Harnack}}\label{sec:proof Harnack}

As is indicated at the end of Section~\ref{sss:main-harnack}, the proof of Theorem~\ref{thm:Harnack} is based on~\cite[Theorem~2.15]{LelRamRey}, which is stated for the Langevin process~\eqref{eq:Langevin_intro}, with smooth coefficients, and in a domain $\mathcal{O}$ which is bounded. Therefore, the main issues to handle are:
\begin{itemize}
    \item[(i)] the more general form of the kinetic SDE~\eqref{eq:Langevin},
    \item[(ii)] the possible low regularity of the coefficients $F$ and $\sigma$,
    \item[(iii)] the possible unboundedness of the domain $\mathcal{O}$.
\end{itemize}
We address the point~(iii) by a localization procedure which is first detailed in Section~\ref{ss:loc-coeffs}. To handle the point~(ii), we also need to introduce a regularization of $F$ and $\sigma$, which is the object of Sections~\ref{ss:loc-reg-coeffs} and~\ref{ss:cv-un-u}. In these two sections, we resort to proofs which are detailed in~\cite{LelRamRey} for the Langevin process, and only emphasize the parts which need to be modified to take into account the point~(i).

The regularization step actually allows to prove the Harnack inequality for a function of the form $u_f(t,x) = \mathbb{E}_x[\ind{t < \tau_\partial}f(X_t)]$, with a continuous function $f$, rather than $u_A(t,x) = \mathbb{P}_x(t < \tau_\partial, X_t \in A)$. A preliminary result regarding $u_f$ is established in Section~\ref{ss:reg-ind}, and the proof of the Harnack inequality for $u_A$ once it has been established for $u_f$ is detailed in Section~\ref{ss:unreg-ind}, thereby completing the proof of Theorem~\ref{thm:Harnack}. The proof of Corollary~\ref{cor:Harnack} is then presented in Section~\ref{ss:pf-cor-harnack}.

\medskip 
Throughout Section~\ref{sec:proof Harnack}, Assumptions~\ref{ass:A1loc}, \ref{ass:A2loc}, \ref{ass:sde} and~\ref{ass:O} are in force.

\subsection{Localization of $F$ and $\sigma$}\label{ss:loc-coeffs}

Let us fix a compact subset $K$ of $D$. As a consequence of Assumption~\ref{ass:O}, it is easily checked that there exists an open, bounded, connected and $\mathcal{C}^2$ set $\mathcal{O}^\flat \subset \R^d$ such that $K \subset D^\flat := \mathcal{O}^\flat \times \R^d$ on the one hand, and $\overline{\mathcal{O}^\flat} \subset \mathcal{O}$ on the other hand. Under Assumptions~\ref{ass:A1loc} and~\ref{ass:A2loc}, there exist functions $F^\flat : \R^{2d} \to \R^d$ and $\sigma^\flat : \R^{2d} \to \R^{d \times d}$ which coincide with $F$ and $\sigma$ on $\overline{D^\flat}$, and which satisfy Assumptions~\ref{ass:A1glob} and~\ref{ass:A2glob}. We denote by $(X^\flat_t)_{t \geq 0}$ the weak solution to~\eqref{eq:Langevin} with coefficients $F^\flat$, $\sigma^\flat$ provided by Proposition~\ref{prop:chaudru}. 

In this step, we show that, as long as the process $X^\flat_t$ remains in $D^\flat$, it has the same law as $X_t$. To proceed, for any $x \in \R^{2d}$, we denote by $\mathrm{P}_x$ (resp. $\mathrm{P}_x^\flat$) the law, on the space of trajectories $\mathcal{C}([0,\infty),\R^{2d})$, of the process $(X_t)_{t \geq 0}$ (resp. $(X^\flat_t)_{t \geq 0}$) with initial condition $x$, provided by Proposition~\ref{prop:chaudru}. As usual, the set $\mathcal{C}([0,\infty),\R^{2d})$ is endowed with the Borel $\sigma$-algebra induced by the topology of locally uniform convergence. Next, we define the functional $\tau^\flat_\partial$ on $\mathcal{C}([0,\infty),\R^{2d})$ by
\begin{equation*}
    \forall \mathrm{x}=(\mathrm{x}_t)_{t \geq 0} \in \mathcal{C}([0,\infty),\R^{2d}), \qquad \tau^\flat_\partial(\mathrm{x}) := \inf\{t > 0: \mathrm{x}_t \not\in D^\flat\},
\end{equation*}
and denote by $\mathcal{G}_{\tau^\flat_\partial}$ the sub-$\sigma$-algebra of $\mathcal{C}([0,\infty),\R^{2d})$ generated by the application $\mathrm{x} \mapsto (\mathrm{x}_{t \wedge \tau^\flat_\partial(\mathrm{x})})_{t \geq 0}$. 

We may now state and prove the main result of this step.

\begin{lemma}[Localization of $F$ and $\sigma$]\label{lem:loc-X}
    For any $x \in \R^{2d}$, the probability measures $\mathrm{P}_x$ and $\mathrm{P}_x^\flat$ coincide on the $\sigma$-algebra $\mathcal{G}_{\tau^\flat_\partial}$.
\end{lemma}
\begin{proof}
    Fix $x \in \R^{2d}$. We construct a coupling of weak solutions to~\eqref{eq:Langevin} with initial condition $x$ with coefficients $F^\flat$, $\sigma^\flat$ and with coefficients $F$, $\sigma$ as follows. Let ${\cal C}:={\cal C}([0,+\infty),\mathbb{R}^{2d})$ be equipped with the canonical filtration. On the space ${\cal C}^2$, we denote by $(X^\flat,X'):=(X^\flat_t,X'_t)_{t\geq 0}$ the canonical process and we consider the probability measure $\mathrm{Q}$ under which $X^\flat$ has distribution $\mathrm{P}^\flat_x$ and, conditionally on $X^\flat$ and given $\{\tau^\flat_\partial<+\infty\}$, $X'$ has distribution $\mathrm{P}_{X^\flat_{\tau^\flat_\partial}}$, where we denote $\tau^\flat_\partial:=\tau^\flat_\partial(X^\flat)$ for convenience. Given $\{\tau^\flat_\partial=+\infty\}$, we take an arbitrary distribution for $X'$, say $\mathrm{P}_0$. Then, under $\mathrm{Q}$, $X^\flat$ is a weak solution to~\eqref{eq:Langevin} with coefficients $F^\flat$, $\sigma^\flat$ and initial condition $x$. In addition, it is  easily checked that the process
    \begin{equation*}
        X_t := \begin{cases}
            X^\flat_t & \text{if $t \leq \tau^\flat_\partial$,}\\
            X'_{t-\tau^\flat_\partial} & \text{otherwise}
        \end{cases}
    \end{equation*}
    is a weak solution to~\eqref{eq:Langevin} with coefficients $F$, $\sigma$ and initial condition $x$, so it has law $\mathrm{P}_x$. Moreover, for any event $A \in \mathcal{G}_{\tau^\flat_\partial}$, it is clear that $X \in A$ if and only if $X^\flat \in A$, which implies that $\mathrm{P}_x(A)=\mathrm{P}^\flat_x(A)$.
\end{proof}

Since the application $\tau^\flat_\partial$ is $\mathcal{G}_{\tau^\flat_\partial}$-measurable, Lemma~\ref{lem:loc-X} shows that $\tau^\flat_\partial(X)$ and $\tau^\flat_\partial(X^\flat)$ have the same law. Therefore, from now on, we shall use the notation $\tau^\flat_\partial$ to refer indifferently to either $\tau^\flat_\partial(X)$ or $\tau^\flat_\partial(X^\flat)$.

\medskip
The following technical results on the small time behavior of the process, which will be used in the sequel, are corollaries of Lemma~\ref{lem:loc-X}.

\begin{lemma}[Localization estimates]\label{lem:loc-estim}
    With the notation introduced above,
    \begin{equation}\label{eq:loc-estim:1}
        \lim_{t \to 0} \sup_{x \in K} \mathbb{P}_x(t \geq \tau^\flat_\partial) = 0,
    \end{equation}
    and, for any $\delta>0$,
    \begin{equation}\label{eq:loc-estim:2}
        \lim_{t \to 0} \sup_{x \in K} \mathbb{P}_x\left(\sup_{s \in [0,t]} |X_s-x| \geq \delta\right) = 0.
    \end{equation}
\end{lemma}
\begin{proof}
    We first establish short-time growth estimates on the process $X^\flat$. For any $s \geq 0$,
    \begin{align*}
        |q^\flat_s-q^\flat_0| \leq \int_0^s |p^\flat_r|\dd r \leq \int_0^s |p^\flat_r-p^\flat_0|\dd r + s |p^\flat_0|,
    \end{align*}
    and by Remark~\ref{rk:linear-growth},
    \begin{align*}
        |p^\flat_s-p^\flat_0| &\leq \int_0^s |F^\flat(q^\flat_r,p^\flat_r)|\dd r + \left|\int_0^s \sigma^\flat(q^\flat_r,p^\flat_r)\dd B_r\right|\\
        &\leq \int_0^s (a^\flat + b^\flat(|q^\flat_r|+|p^\flat_r|))\dd r + \left|\int_0^s \sigma^\flat(q^\flat_r,p^\flat_r)\dd B_r\right|\\
        &\leq b^\flat \int_0^s (|q^\flat_r-q^\flat_0|+|p^\flat_r-p^\flat_0|)\dd r + s\left(a^\flat + b^\flat(|q^\flat_0|+|p^\flat_0|)\right) + \left|\int_0^s \sigma^\flat(q^\flat_r,p^\flat_r)\dd B_r\right|.
    \end{align*}
    As a consequence, for $X^\flat_0=x \in K$,
    \begin{align*}
        |X^\flat_s-x| &= |q^\flat_s-q^\flat_0| + |p^\flat_s-p^\flat_0|\\
        &\leq (b^\flat+1)\int_0^s |X^\flat_r-x| \dd r + C_Ks + \left|\int_0^s \sigma^\flat(q^\flat_r,p^\flat_r)\dd B_r\right|,
    \end{align*}
    where $C_K := \sup\{|p|+a^\flat+b^\flat(|q|+|p|), (q,p) \in K\}$, so by Gronwall's lemma, for any $t \geq 0$,
    \begin{equation}\label{eq:pf-loc-estim:1}
        \sup_{s \in [0,t]} |X^\flat_s-x| \leq \mathrm e^{(b^\flat+1)t}\left(C_K t + \sup_{s \in [0,t]} \left|\int_0^s \sigma^\flat(q^\flat_r,p^\flat_r)\dd B_r\right|\right).
    \end{equation}

    We now use the estimate~\eqref{eq:pf-loc-estim:1} to prove~\eqref{eq:loc-estim:1}. Let us set
    \begin{equation*}
        \delta := \inf\{|q'-q|, (q,p) \in K, q' \in \partial\mathcal{O}^\flat\} > 0,
    \end{equation*}
    so that, for any $t \geq 0$ and $x=(q,p) \in K$,
    \begin{align*}
        \mathbb{P}_x\left(t \geq \tau^\flat_\partial\right) &\leq \mathbb{P}_x\left(\sup_{s \in [0,t]} |q^\flat_s-q| \geq \delta\right)\\
        &\leq \mathbb{P}_x\left(\sup_{s \in [0,t]} |X^\flat_s-x| \geq \delta\right)\\
        &\leq \mathbb{P}_x\left(\mathrm e^{(b^\flat+1)t}\left(C_K t + \sup_{s \in [0,t]} \left|\int_0^s \sigma^\flat(q^\flat_r,p^\flat_r)\dd B_r\right|\right) \geq \delta\right)\\
        &\leq \delta^{-1}\mathrm e^{(b^\flat+1)t}\left(C_K t + \mathbb{E}_x\left[\sup_{s \in [0,t]} \left|\int_0^s \sigma^\flat(q^\flat_r,p^\flat_r)\dd B_r\right|\right]\right),
    \end{align*}
    where we have used~\eqref{eq:pf-loc-estim:1} and the Markov inequality. The Cauchy--Schwarz inequality, the Doob maximal inequality and Ito's isometry then yield
    \begin{align*}
        \mathbb{E}_x\left[\sup_{s \in [0,t]} \left|\int_0^s \sigma^\flat(q^\flat_r,p^\flat_r)\dd B_r\right|\right] &\leq \sqrt{\mathbb{E}_x\left[\sup_{s \in [0,t]} \left|\int_0^s \sigma^\flat(q^\flat_r,p^\flat_r)\dd B_r\right|^2\right]}\\
        &\leq \sqrt{4\mathbb{E}_x\left[\left|\int_0^t \sigma^\flat(q^\flat_r,p^\flat_r)\dd B_r\right|^2\right]} \leq 2\sqrt{t c_2^\flat},
    \end{align*}
    where $c_2^\flat$ is given by Assumption~\ref{ass:A2glob} for $\sigma^\flat$. Since the right-hand side does not depend on $x$ and vanishes with $t$, the conclusion follows.

    We finally prove~\eqref{eq:loc-estim:2}. To proceed, we fix $t \geq 0$, $x \in K$, $\delta>0$, and write
    \begin{align*}
        \mathbb{P}_x\left(\sup_{s \in [0,t]} |X_s-x| \geq \delta\right)  &\leq \mathbb{P}_x\left(\sup_{s \in [0,t]} |X_s-x| \geq \delta, t < \tau^\flat_\partial\right) + \mathbb{P}_x\left(\sup_{s \in [0,t]} |X_s-x| \geq \delta, t \geq \tau^\flat_\partial\right).
    \end{align*}
    On the one hand, by~\eqref{eq:loc-estim:1}, 
    \begin{equation*}
        \lim_{t \to 0} \sup_{x \in K} \mathbb{P}_x\left(\sup_{s \in [0,t]} |X_s-x| \geq \delta, t \geq \tau^\flat_\partial\right) = 0.
    \end{equation*}
    On the other hand, by Lemma~\ref{lem:loc-X},
    \begin{align*}
        \mathbb{P}_x\left(\sup_{s \in [0,t]} |X_s-x| \geq \delta, t < \tau^\flat_\partial\right) = \mathbb{P}_x\left(\sup_{s \in [0,t]} |X^\flat_s-x| \geq \delta, t < \tau^\flat_\partial\right) \leq \mathbb{P}_x\left(\sup_{s \in [0,t]} |X^\flat_s-x| \geq \delta\right),
    \end{align*}
    and by the same arguments as in the proof of~\eqref{eq:loc-estim:1}, the right-hand side vanishes with $t$, uniformly in $x \in K$.
\end{proof}

\subsection{Regularization of $\ind{x \in A}$}\label{ss:reg-ind}

For Sections~\ref{ss:reg-ind} to~\ref{ss:cv-un-u}, we let $f : D \to \R$ be a continuous and bounded function, and set
\begin{equation}\label{eq:def-uf}
    \forall t \geq 0, \quad x \in D, \qquad u_f(t,x) := \mathbb{E}_x\left[\ind{t<\tau_\partial}f(X_t)\right].
\end{equation}
Only in Section~\ref{ss:unreg-ind} we will replace $f$ with the possibly discontinuous function $\ind{x \in A}$. For now, working with a continuous function $f$ enables us to prove the following statement.

\begin{lemma}[Continuity of $u_f$]\label{lem:cont-uf}
    The function $u_f$ defined in~\eqref{eq:def-uf} is continuous on $[0,\infty) \times D$.
\end{lemma}
\begin{proof}
    We shall proceed differently to show the continuity of $u_f$ at points $(t,x) \in [0,\infty) \times D$, depending on whether $t=0$ or $t>0$. However, both arguments rely on the preliminary result that for any compact subset $K \subset D$,
    \begin{equation}\label{eq:pf-cont-1}
        \lim_{t \to 0} \sup_{x \in K} \mathbb{P}_x(\tau_\partial \leq t) = 0.
    \end{equation}
    This claim immediately follows from~\eqref{eq:loc-estim:1}, since, with the notation of Section~\ref{ss:loc-coeffs}, $\tau_\partial \geq \tau^\flat_\partial$.

    \medskip
    \emph{Continuity at $(t,x)$ with $t=0$.} We fix a compact set $K \subset D$ and show that
    \begin{equation*}
        \lim_{t \to 0} \sup_{x \in K} |u_f(t,x)-f(x)| = 0.
    \end{equation*}
    Since $f$ is continuous on $D$, this shows in particular that the function $u_f$, defined on $[0,\infty) \times D$, is continuous at any point of the form $(0,x)$, with $x \in D$. To proceed, we first write, for any $t \geq 0$, $x \in K$,
    \begin{equation*}
        |u_f(t,x)-f(x)| \leq \mathbb{E}_x\left[\ind{t < \tau_\partial}|f(X_t)-f(x)|\right] + \|f\|_\infty \mathbb{P}_x(\tau_\partial \leq t),
    \end{equation*}
    and observe that by~\eqref{eq:pf-cont-1}, the second term vanishes with $t$, uniformly in $x \in K$. Now let $\epsilon>0$ and let $\alpha>0$ small enough such that the compact set $K_\alpha := \{y \in \R^{2d}: \mathrm{d}(y,K) \leq \alpha\}$ is included in $D$. Since $f$ is uniformly continuous on the compact set $K_\alpha$, there exists $\delta>0$ such that for any $x,y \in K_\alpha$, if $|x-y| \leq \delta$ then $|f(x)-f(y)|\leq \epsilon$, and without loss of generality, $\delta$ can be taken smaller than $\alpha$. As a consequence, 
    \begin{align*}
        \mathbb{E}_x\left[\ind{t < \tau_\partial}|f(X_t)-f(x)|\right] &= \mathbb{E}_x\left[\ind{t < \tau_\partial, |X_t-x| \leq \delta}|f(X_t)-f(x)|\right]\\
        &\quad + \mathbb{E}_x\left[\ind{t < \tau_\partial, |X_t-x| > \delta}|f(X_t)-f(x)|\right]\\
        &\leq \epsilon + 2\|f\|_\infty \mathbb{P}_x(|X_t-x|>\delta),
    \end{align*}
    so the conclusion follows from~\eqref{eq:loc-estim:2}.
    
    \medskip
    \emph{Continuity at $(t,x)$ with $t>0$.} We now fix $t>0$, $x \in D$ and show that $u_f$ is continuous at $(t,x)$. To proceed, we let $(t_n,x_n)_{n \geq 1}$ be a sequence of elements of $[0,\infty) \times D$ which converges to $(t,x)$, and write
    \begin{equation*}
        |u_f(t_n,x_n)-u_f(t,x)| \leq |u_f(t_n,x_n)-u_f(t_n,x)| + |u_f(t_n,x)-u_f(t,x)|.
    \end{equation*}
    First,
    \begin{equation*}
        u_f(t_n,x) = \mathbb{E}_x\left[\ind{t_n < \tau_\partial}f(X_{t_n})\right].
    \end{equation*}
    Since the trajectory of $(X_t)_{t \geq 0}$ is continuous, almost surely, and $f$ is continuous, $f(X_{t_n}) \to f(X_t)$, almost surely. Moreover, by Assumptions~\ref{ass:O} and~\ref{ass:sde},
    \begin{equation*}
        \mathbb{P}_x(t=\tau_\partial) \leq \mathbb{P}_x(q_t \in \partial \mathcal{O}) = \int_{\partial \mathcal{O}\times\R^d} p(t;x,y)\dd y = 0,
    \end{equation*}
    therefore $\ind{t_n < \tau_\partial} \to \ind{t < \tau_\partial}$, almost surely. Since $f$ is bounded, we deduce from the dominated convergence theorem that
    \begin{equation*}
        \lim_{n \to \infty} |u_f(t_n,x)-u_f(t,x)| = 0.
    \end{equation*}
    Now let us fix $s \in (0,t)$. For $n$ large enough, we have $t_n \geq s$ and thus, by the Markov property,
    \begin{equation*}
        u_f(t_n,x_n) = \mathbb{E}_{x_n}\left[\ind{s < \tau_\partial} u_f(t_n-s,X_s)\right], \qquad u_f(t_n,x) = \mathbb{E}_x\left[\ind{s < \tau_\partial} u_f(t_n-s,X_s)\right],
    \end{equation*}
    so that
    \begin{equation*}
        |u_f(t_n,x_n)-u_f(t_n,x)| \leq \left|\mathbb{E}_{x_n}\left[u_f(t_n-s,X_s)\right]-\mathbb{E}_x\left[u_f(t_n-s,X_s)\right]\right| + \|f\|_\infty \left(\mathbb{P}_{x_n}(\tau_\partial \leq s) + \mathbb{P}_x(\tau_\partial \leq s)\right)
    \end{equation*}
    By~\eqref{eq:pf-cont-1}, we have
    \begin{equation*}
        \lim_{s \to 0} \limsup_{n \to \infty} \mathbb{P}_{x_n}(\tau_\partial \leq s) + \mathbb{P}_x(\tau_\partial \leq s) = 0.
    \end{equation*}
    On the other hand, by Assumption~\ref{ass:sde}, for fixed $s$,
    \begin{equation*}
        \mathbb{E}_{x_n}\left[u_f(t_n-s,X_s)\right]-\mathbb{E}_x\left[u_f(t_n-s,X_s)\right] = \int_{\R^{2d}} u_f(t_n-s,y)\left(p(s;x_n,y)-p(s;x,y)\right)\dd y,
    \end{equation*}
    where the transition density $p(s;x,y)$ of $(X_t)_{t \geq 0}$ is continuous in $x$. Therefore, by Scheffé's lemma, since $u_f(t_n-s,y)$ is bounded uniformly in $n$, 
    \begin{equation*}
        \lim_{n \to \infty} \int_{\R^{2d}} u_f(t_n-s,y)\left(p(s;x_n,y)-p(s;x,y)\right)\dd y = 0,
    \end{equation*}
    which concludes the proof. 
\end{proof}

\subsection{Regularization of $F^\flat$ and $\sigma^\flat$}\label{ss:loc-reg-coeffs}

From now on we fix a compact set $K \subset D$ and use the notation introduced in Section~\ref{ss:loc-coeffs}. By the strong Markov property, for any $t \geq 0$, $x \in D$,
\begin{align*}
    u_f(t,x) &= \mathbb{E}_x\left[\ind{\tau_\partial > t}f(X_t) \right]\\
    &= \mathbb{E}_x\left[\ind{\tau_\partial > t, \tau^\flat_\partial > t}f(X_t) \right] + \mathbb{E}_x\left[\ind{\tau_\partial > t, \tau^\flat_\partial \leq t}f(X_t) \right]\\
    &= \mathbb{E}_x\left[\ind{\tau^\flat_\partial > t}f(X_t) \right] + \mathbb{E}_x\left[\ind{\tau^\flat_\partial \leq t}u_f(t-\tau^\flat_\partial,X_{\tau^\flat_\partial})\right],
\end{align*}
and by Lemma~\ref{lem:loc-X}, one may replace $(X_t)_{t \geq 0}$ with $(X^\flat_t)_{t \geq 0}$ in the right-hand side and thus end up with the identity
\begin{equation*}
    u_f(t,x) = \mathbb{E}_x\left[\ind{\tau^\flat_\partial > t}f(X^\flat_t) + \ind{\tau^\flat_\partial \leq t}g(t-\tau^\flat_\partial,X^\flat_{\tau^\flat_\partial})\right],
\end{equation*}
where we have somehow artificially introduced the boundary condition 
\begin{equation*}
    g(t,x):=u_f(t,x),    
\end{equation*}
which by Lemma~\ref{lem:cont-uf} is continuous on $[0,\infty) \times \overline{D^\flat}$. If the functions $F^\flat$ and $\sigma^\flat$ are $\mathcal{C}^\infty$, and globally bounded and Lipschitz continuous, then following the proof of~\cite[Theorem 2.10 (iii)]{LelRamRey} (combined with Remark~2.12 in the same reference), one deduces from this identity that $u_f$ is a classical solution to the kinetic Fokker--Planck equation
\begin{equation}\label{eq:kFP}
    \left\{\begin{aligned}
        \partial_t u_f(t,x) &= \mathcal{L}_{F^\flat,\sigma^\flat} u_f(t,x), & (t,x) \in (0,\infty) \times D^\flat,\\
        u_f(0,x) &= f(x), & x \in D^\flat,\\
        u_f(t,x) &= g(t,x) & \text{$t>0$, $x=(q,p)$ with $q \in \partial \mathcal{O}^\flat$, $p \cdot n^\flat(q) \geq 0$,}
    \end{aligned}\right.
\end{equation}
where $n^\flat(q)$ denotes the outward normal vector at $q \in \partial \mathcal{O}^\flat$. In particular, $u_f$ is a distributional solution to $\partial_t u_f = \mathcal{L}_{F^\flat,\sigma^\flat} u_f$, so as soon as $f \geq 0$, the Harnack inequality for $u_f$ follows from the following statement.

\begin{proposition}[Harnack inequality for smooth coefficients]\label{prop:Harnack-smooth}
    Assume that the functions $F^\flat$ and $\sigma^\flat$ are $\mathcal{C}^\infty$. For any $\epsilon>0$ and $T>0$, there exists a constant $C_{K,\epsilon,T} \geq 0$ which only depends on $F^\flat$ and $\sigma^\flat$ through $\sup_K |F^\flat|$, $\sup_K \|\sigma^\flat{\sigma^\flat}^T\|$ and the uniform ellipticity constant of $\sigma^\flat{\sigma^\flat}^T$ on $K$, such that, for any nonnegative $L^1_\mathrm{loc}$ function $u : (0,\infty) \times D^\flat \to [0,\infty)$ which is a distributional solution to $\partial_t u = \mathcal{L}_{F^\flat,\sigma^\flat}u$ on $(0,\infty) \times D^\flat$, one has
    \begin{equation*}
        \forall t \geq \epsilon, \qquad \sup_{x \in K} u(t,x) \leq C_{K,\epsilon,T} \inf_{x \in K} u(t+T,x).
    \end{equation*}
\end{proposition}

\begin{remark}
    In the statement of Proposition~\ref{prop:Harnack-smooth}, a hypoellipticity argument ensures that as soon as $u \in L^1_\mathrm{loc}$ is a distributional solution to $\partial_t u = \mathcal{L}_{F^\flat,\sigma^\flat}u$ then it is $\mathcal{C}^\infty$, so it makes sense to consider pointwise values of $u(t,x)$.
\end{remark}

Proposition~\ref{prop:Harnack-smooth} is stated and proved in~\cite[Theorem 2.15]{LelRamRey} for the particular case $F^\flat(q,p) = F(q)-\gamma p$, $\sigma^\flat(q,p) = \sigma I_d$. It combines two building bricks: a Harnack inequality for kinetic Fokker--Planck equations in small balls, initially obtained by Golse, Imbert, Mouhot and Vasseur~\cite[Theorem~4]{Har}, and a chaining argument to extend the inequality to arbitrary compact sets, adapted from~\cite{AncPol20}. Both steps can be adapted without any modification to the setting of Proposition~\ref{prop:Harnack-smooth}, so we omit the details of the proof.

In view of Proposition~\ref{prop:Harnack-smooth}, the main issue with the setting of Theorem~\ref{thm:Harnack} is the possible low regularity of $F^\flat$ and $\sigma^\flat$. So we introduce regularized versions thereof, defined as follows: for $n \ge 1$, for $x\in\R^{2d}$,
\begin{equation}\label{eq:def F_n sigma_n}
    F^\flat_n(x)=\int_{\R^{2d}}F^\flat(x-y)\phi_n(y)\,\mathrm dy,\qquad \sigma^\flat_n(x)=\int_{\R^{2d}}\sigma^\flat(x-y)\phi_n(y)\,\mathrm dy,
\end{equation} 
where $\phi_n(y)=n^{2d}\phi(ny)$ with $\phi$ a $\mathcal{C}^\infty$ and compactly supported function on $\R^{2d}$ integrating to $1$. With this definition we get the following statement, which follows from standard arguments and whose proof is omitted.
\begin{lemma}[Regularized coefficients]\label{lem:cv F_n sigma_n}
  One may choose the function $\phi$ such that the following hold:
  \begin{enumerate}[label=(\roman*),ref=\roman*]
    \item We have 
    \begin{equation}\label{eq:cv-F_n}
      F^\flat_n\overset{L^1_{\mathrm{loc}}}{\underset{n\rightarrow\infty}{\longrightarrow}}F^\flat,
    \end{equation}
    and there exist $\tilde{a}^\flat, \tilde{b}^\flat \geq 0$ such that 
    \begin{equation}\label{eq:uniform_linear_growth}
      \forall x \in \R^{2d}, \quad n \geq 1, \qquad |F^\flat_n(x)| \leq \tilde{a}^\flat+\tilde{b}^\flat|x|.
    \end{equation}
  \item We have
    \begin{equation}\label{eq:unif conv sigma}
      \Vert \sigma^\flat_n-\sigma^\flat\Vert_{\infty}\underset{n\rightarrow\infty}{\longrightarrow}0,
    \end{equation}
    and there exist $0 < \tilde{c}^\flat_1 \leq \tilde{c}^\flat_2$ such that for all $x \in \R^{2d}$, $\xi \in \R^d$, for all $n \geq 1$, 
    \begin{equation}\label{eq:tildec1c2}
      \tilde{c}^\flat_1 |\xi|^2 \leq \xi \cdot \sigma^\flat_n{\sigma^\flat_n}^T(x)\xi \leq \tilde{c}^\flat_2|\xi|^2.
    \end{equation}
    Moreover, for all $x = (q,p), x'=(q',p') \in \R^{2d}$, for all $n \geq 1$,
    \begin{equation*}
        \Vert \sigma^\flat_n(x)-\sigma^\flat_n(x')\Vert\leq c^\flat_3(\vert q-q'\vert^{\alpha/3}+\vert p-p'\vert^\alpha),
    \end{equation*}
    with $c^\flat_3$ and $\alpha$ given by Assumption~\ref{ass:A2glob} for $\sigma^\flat$.
  \end{enumerate}
\end{lemma}

Since, by Lemma~\ref{lem:cv F_n sigma_n}, $F^\flat_n$ and $\sigma^\flat_n$ are locally Lipschitz continuous, $F^\flat_n$ is of linear growth and $\sigma^\flat_n{\sigma^\flat_n}^T$ is bounded, with coefficients independent of $n$, there exists a unique strong solution $(X^{\flat,n}_t)_{t \ge 0}$ to~\eqref{eq:Langevin} with coefficients $F^\flat_n$ and $\sigma^\flat_n$, see~\cite[Theorem 2.2, p.105]{F}. In addition, this solution satisfies for all $t\geq0$ and for all $x\in\R^{2d}$,
\begin{equation}\label{eq:order 2 moment}
    \mathbb{E}_x\left[\sup_{s\in[0,t]}\vert X^{\flat,n}_s\vert^2\right]\leq \tilde{C}^\flat,
\end{equation}
where $\tilde{C}^\flat$ is a constant only depending on $t,x$ and the constants $\tilde{a}^\flat$, $\tilde{b}^\flat$ and $\tilde{c}^\flat_2$ from Lemma~\ref{lem:cv F_n sigma_n}. Let us define $\tau^{\flat,n}_\partial := \inf\{t > 0: X^{\flat,n}_t \not\in D^\flat\}$, and set, for any $t \geq 0$ and $x \in D^\flat$,
\begin{equation*}
    u^n_f(t,x) = \mathbb{E}_x\left[\ind{\tau^{\flat,n}_\partial > t}f\left(X^{\flat,n}_t\right) + \ind{\tau^{\flat,n}_\partial \leq t}g\left(t-\tau^{\flat,n}_\partial,X^{\flat,n}_{\tau^{\flat,n}_\partial}\right)\right].
\end{equation*}

\begin{proposition}\label{prop:u_n Harnack}
    For any $n \geq 1$, $\partial_t u^n_f = \mathcal{L}_{F_n^\flat,\sigma_n^\flat}u^n_f$ in the distributional sense on $(0,\infty) \times D^\flat$.

    As a consequence, for any $\epsilon>0$ and $T>0$, there exists a constant $C_{K,\epsilon,T} \geq 0$ which depends neither on $f$ nor on $n$ such that, as soon as $f \geq 0$,
    \begin{equation*}
        \forall t \geq \epsilon, \qquad \sup_{x \in K} u^n_f(t,x) \leq C_{K,\epsilon,T} \inf_{x \in K} u^n_f(t+T,x).
    \end{equation*}\end{proposition}
\begin{proof}
    It is clear that the second part of the statement follows from the combination of Lemma~\ref{lem:cv F_n sigma_n} with Proposition~\ref{prop:Harnack-smooth}, so we focus on the first part. 

    By construction, the functions $F^\flat_n$ and $\sigma^\flat_n$ are $\mathcal{C}^\infty$. If they were Lipschitz continuous and bounded, the statement of the Proposition would follow from the same reasoning as in the proof of~\cite[Theorem 2.10 (iii) and Remark~2.12]{LelRamRey}. In order to apply this argument to our setting, let us define
    \begin{equation*}
        F^\flat_{n,k}=F^\flat_n e_k,\qquad\sigma^\flat_{n,k}=\sigma^\flat_n e_k,
    \end{equation*}
    where $e_k$ is a $\mathcal{C}^\infty$, compactly supported function on $\R^{2d}$, which satisfies $e_k(x)=1$ for $|x| \leq k.$ Let us then denote by $(X^{\flat,n,k}_t)_{t\geq0}=(q^{\flat,n,k}_t,p^{\flat,n,k}_t)_{t\geq0}$ the unique strong solution to~\eqref{eq:Langevin} with the coefficients $F^\flat_{n,k}$ and $\sigma^\flat_{n,k}$. Let $\tau_\partial^{\flat,n,k}$ be its first exit time from $D^\flat$. Following the arguments of~\cite[Theorem 2.10 (iii) and Remark~2.12]{LelRamRey}, the function
    \begin{equation*}
        u^{n,k}_f(t,x) := \mathbb{E}_x\left[\ind{\tau^{\flat,n,k}_\partial > t}f\left(X^{\flat,n,k}_t\right) + \ind{\tau^{\flat,n,k}_\partial \leq t}g\left(t-\tau^{\flat,n,k}_\partial,X^{\flat,n,k}_{\tau^{\flat,n,k}_\partial}\right)\right]
    \end{equation*}
    is a classical solution to the Initial-Boundary Value Problem~\eqref{eq:kFP}, with differential operator $\mathcal{L}_{F^\flat_{n,k},\sigma^\flat_{n,k}}$. In particular, we have
    \begin{equation*}
        \partial_t u^{n,k}_f = \mathcal{L}_{F_{n,k}^\flat,\sigma_{n,k}^\flat}u^{n,k}_f
    \end{equation*}
    in the distributional sense on $(0,\infty) \times D^\flat$. Therefore, to conclude the proof, it suffices to show that for any $(t,x) \in (0,\infty) \times D^\flat$,
    \begin{equation}\label{eq:lim-unk}
        \lim_{k \to \infty} u^{n,k}_f(t,x) = u^n_f(t,x).
    \end{equation}
    
    Let 
    \begin{equation*}
        \tau^{\flat,n,k}_k:=\inf\left\{t>0:\vert X^{\flat,n,k}_t\vert>k\right\},\qquad\tau^{\flat,n}_k:=\inf\left\{t>0:\vert X^{\flat,n}_t\vert>k\right\},
    \end{equation*} 
    then for $t>0$, $x\in D^\flat$, 
    \begin{align*}
        u_f^{n,k}(t,x) &=\mathbb{E}_x\left[\left(\ind{\tau^{\flat,n,k}_\partial > t}f(X^{\flat,n,k}_t)  + \ind{\tau^{\flat,n,k}_\partial \leq t}g(t-\tau^{\flat,n,k}_\partial,X^{\flat,n,k}_{\tau^{\flat,n,k}_\partial})\right)\ind{\tau^{\flat,n,k}_k>t}\right]\\
        &\quad + \mathbb{E}_x\left[\left(\ind{\tau^{\flat,n,k}_\partial > t}f(X^{\flat,n,k}_t)  + \ind{\tau^{\flat,n,k}_\partial \leq t}g(t-\tau^{\flat,n,k}_\partial,X^{\flat,n,k}_{\tau^{\flat,n,k}_\partial})\right)\ind{\tau^{\flat,n,k}_k \leq t}\right].
    \end{align*}
    Since $F^\flat_{n,k}=F^\flat_n$ and $\sigma^\flat_{n,k}=\sigma^\flat_n$ on the ball of radius $k$, by~\cite[Theorem 5.2.1]{F}, the processes $(X^{\flat,n,k}_s)_{s\in[0,t]}$ and $(X^{\flat,n}_s)_{s\in[0,t]}$ coincide almost-surely on the event $\{\tau_k^{\flat,n,k}> t\}$, and $\tau_k^{\flat,n,k}=\tau_k^{\flat,n}$ almost-surely. Therefore,
    \begin{equation}\label{eq:egalite en loi}
        \begin{aligned}
            &\mathbb{E}_x\left[\left(\ind{\tau^{\flat,n,k}_\partial > t}f(X^{\flat,n,k}_t)  + \ind{\tau^{\flat,n,k}_\partial \leq t}g(t-\tau^{\flat,n,k}_\partial,X^{\flat,n,k}_{\tau^{\flat,n,k}_\partial})\right)\ind{\tau^{\flat,n,k}_k>t}\right]\\
            & = \mathbb{E}_x\left[\left(\ind{\tau^{\flat,n}_\partial > t}f(X^{\flat,n}_t)  + \ind{\tau^{\flat,n}_\partial \leq t}g(t-\tau^{\flat,n}_\partial,X^{\flat,n}_{\tau^{\flat,n}_\partial})\right)\ind{\tau^{\flat,n}_k>t}\right].
        \end{aligned}
    \end{equation}
    On the other hand, \eqref{eq:order 2 moment} implies that $\tau^{\flat,n}_k \to \infty$ when $k$ goes to infinity, almost surely. Thus, since $f$ and $g$ are bounded, we conclude from the dominated convergence theorem that~\eqref{eq:lim-unk} holds.
\end{proof}  

\subsection{Pointwise convergence of $u^n_f$ toward $u_f$}\label{ss:cv-un-u}

Now that the Harnack inequality is established for the function $u^n_f$ defined before Proposition~\ref{prop:u_n Harnack}, in the present section we remove the regularization of the coefficients $F^\flat$, $\sigma^\flat$ by proving the following statement.

\begin{proposition}[Convergence of $u^n_f$ toward $u_f$]\label{prop:cv-un}
    For any $t > 0$ and $x \in D^\flat$, 
    \begin{equation*}
        \lim_{n \to \infty} u^n_f(t,x) = u_f(t,x).
    \end{equation*}
\end{proposition}

The proof of Proposition~\ref{prop:cv-un} is based on the following two lemmas.

\begin{lemma}[Weak convergence]\label{lem:weak convergence}
For all $x\in\R^{2d}$, $t>0$, the strong solution $(X^{\flat,n}_s)_{s\in[0,t]}$ to~\eqref{eq:Langevin} with coefficients $F^\flat_n$, $\sigma^\flat_n$ and initial condition $x$, converges in distribution on $(\mathcal C([0,t],\mathbb{R}^{2d}),\|\cdot\|_\infty)$ to the weak solution $(X^\flat_s)_{s\in[0,t]}$ to~\eqref{eq:Langevin} with coefficients $F^\flat$, $\sigma^\flat$ and initial condition $x$, when $n$ goes to infinity.
\end{lemma}

\begin{lemma}[Instant exit from $\overline{D^\flat}$ when $X^\flat_0 \in \Gamma^{\flat,0}$]\label{lem:gamma^0}
Let $$\Gamma^{\flat,0}=\{ (q,p)\in\partial\mathcal{O}^\flat \times\mathbb{R}^d : p\cdot n^\flat(q)=0 \},$$
where $n^\flat(q)$ is the outward normal vector at $q \in \partial \mathcal{O}^\flat$. For all $x\in\Gamma^{\flat,0}$,
\begin{equation*}
    \mathbb{P}_x\left(\forall \epsilon > 0, \exists t \in [0,\epsilon]: X^\flat_t \not\in \overline{D^\flat}\right)=1.
\end{equation*}
\end{lemma}

Taking the statement of Lemmas~\ref{lem:weak convergence} and~\ref{lem:gamma^0} for granted, we first present the proof of Proposition~\ref{prop:cv-un}. We then detail the proofs of Lemmas~\ref{lem:weak convergence} and~\ref{lem:gamma^0}.

\begin{proof}[Proof of Proposition~\ref{prop:cv-un}]
The reasoning is similar to~\cite[Lemma 3.5]{Ram}. Let $t>0$ and define the function 
\begin{equation*}
    \mathcal{S}_t : (\mathrm{x}_s)_{s \in [0,t]} \in \mathcal{C}([0,t],\R^{2d}) \mapsto \ind{\tau^\flat_\partial > t}f(\mathrm{x}_t) + \ind{\tau^\flat_\partial \leq t} g(t-\tau^\flat_\partial,\mathrm{x}_{\tau^\flat_\partial}),
\end{equation*}
with the convention that $\tau^\flat_\partial=\infty$ if for all $s \in [0,t]$, $\mathrm{x}_s \in D^\flat$. We then have, for any $x \in D$,
\begin{equation*}
    u^n_f(t,x) = \mathbb{E}_x\left[\mathcal{S}_t\left((X^{\flat,n}_s)_{s \in [0,t]}\right)\right], \qquad u_f(t,x) = \mathbb{E}_x\left[\mathcal{S}_t\left((X^\flat_s)_{s \in [0,t]}\right)\right].
\end{equation*}
In view of Lemma~\ref{lem:weak convergence} and the boundedness of $\mathcal{S}_t$, by the continuous mapping theorem it suffices to 
construct a measurable subset $\mathcal{A}$ of $\mathcal{C}([0,t],\R^{2d})$ such that $\mathbb{P}_x((X^\flat_s)_{s \in [0,t]} \in \mathcal{A})=1$ and $\mathcal{S}_t$ is continuous at each element of $\mathcal{A}$ in order to conclude that $u^n_f(t,x)$ converges to $u_f(t,x)$. 

We shall prove that these properties hold for $\mathcal{A} = \mathcal{A}_\mathrm{in} \cup \mathcal{A}_\mathrm{out}$, where
\begin{equation*}
    \mathcal{A}_\mathrm{in} = \{(\mathrm{x}_s)_{s \in [0,t]} : \forall s \in [0,t], \mathrm{x}_s \in D^\flat\}
\end{equation*}
is the set of trajectories which remain in $D^\flat$, and
\begin{equation*}
    \mathcal{A}_\mathrm{out} = \{(\mathrm{x}_s)_{s \in [0,t]} : \tau^\flat_\partial < t \text{ and } \forall \epsilon>0, \exists s \in [\tau^\flat_\partial, \tau^\flat_\partial + \epsilon]: \mathrm{x}_s \not\in \overline{D^\flat}\}
\end{equation*}
is the set of trajectories which reach $\partial D^\flat$ and immediately exit $\overline{D^\flat}$.

Let us first check that $\mathcal{S}_t$ is continuous at every trajectory in $\mathcal{A}$. If $(\mathrm{x}_s)_{s \in [0,t]}$ is in the open set $\mathcal{A}_\mathrm{in}$, then $\tau^\flat_\partial = \infty$, and so is the case of any trajectory close enough. Therefore, the continuity of $\mathcal{S}_t$ at $(\mathrm{x}_s)_{s \in [0,t]}$ follows from the continuity of the mapping $(\mathrm{x}_s)_{s \in [0,t]} \mapsto f(\mathrm{x}_t)$. Now if $(\mathrm{x}_s)_{s \in [0,t]} \in \mathcal{A}_\mathrm{out}$, it follows from elementary arguments that $\tau^\flat_\partial$ and $\mathrm{x}_{\tau^\flat_\partial}$ are continuous at $(\mathrm{x}_s)_{s \in [0,t]}$, so that the continuity of $\mathcal{S}_t$ at $(\mathrm{x}_s)_{s \in [0,t]}$ follows from the continuity of $g$.

We now check that for any $x \in D^\flat$, $(X^\flat_s)_{s \in [0,t]} \in \mathcal{A}$, $\mathbb{P}_x$-almost surely. To proceed, since $\tau^\flat_\partial > t$ if and only if $(X^\flat_s)_{s \in [0,t]} \in \mathcal{A}_\mathrm{in}$, it suffices to show that
\begin{equation*}
    \mathbb{P}_x\left(\tau^\flat_\partial \leq t, (X^\flat_s)_{s \in [0,t]} \not\in \mathcal{A}_\mathrm{out}\right) = 0.
\end{equation*}
First, since by Assumption~\ref{ass:sde} and Lemma~\ref{lem:loc-X},
\begin{equation*}
    \mathbb{P}_x(\tau^\flat_\partial=t) \leq \mathbb{P}_x(q_t \in \partial \mathcal{O}^\flat) = \int_{\partial \mathcal{O}^\flat \times \R^d} p(t;x,y)\dd y = 0,
\end{equation*}
it actually suffices to show that
\begin{equation*}
    \mathbb{P}_x\left(\tau^\flat_\partial < t, (X^\flat_s)_{s \in [0,t]} \not\in \mathcal{A}_\mathrm{out}\right) = 0.
\end{equation*}
Now, by the strong Markov property, we have
\begin{equation*}
    \mathbb{P}_x\left(\tau^\flat_\partial < t, (X^\flat_s)_{s \in [0,t]} \not\in \mathcal{A}_\mathrm{out}\right) = \mathbb{E}_x\left[\ind{\tau^\flat_\partial < t}\mathbb{P}_{X^\flat_{\tau^\flat_\partial}}(\exists \epsilon > 0: \forall s \in [0,\epsilon], X^\flat_s \in \overline{D^\flat})\right].
\end{equation*}
The $\mathcal{C}^2$ regularity of $\partial \mathcal{O}^\flat$ entails, by the exterior sphere condition, that on the event $\tau^\flat_\partial < t$, $X^\flat_{\tau^\flat_\partial}$ necessarily belongs to $\Gamma^{\flat,0} \cup \Gamma^{\flat,+}$, where $\Gamma^{\flat,0}$ is introduced in Lemma~\ref{lem:gamma^0}, and we define
\begin{equation*}
    \Gamma^{\flat,+} := \{(q,p) \in \partial\mathcal{O}^\flat \times \R^d: p \cdot n^\flat(q) > 0\},
\end{equation*}
see~\cite[Sect.~3.4.1]{LelRamRey} for details. By the same geometric arguments, if $X^\flat_{\tau^\flat_\partial} \in \Gamma^{\flat,+}$ then it is elementary to show that $\mathbb{P}_{X^\flat_{\tau^\flat_\partial}}(\exists \epsilon > 0: \forall s \in [0,\epsilon], X^\flat_s \in \overline{D^\flat})=0$, see again~\cite[Sect.~3.4.1]{LelRamRey}. On the other hand, if $X^\flat_{\tau^\flat_\partial} \in \Gamma^{\flat,0}$ then the identity $\mathbb{P}_{X^\flat_{\tau^\flat_\partial}}(\exists \epsilon > 0: \forall s \in [0,\epsilon], X^\flat_s \in \overline{D^\flat})=0$ now follows from Lemma~\ref{lem:gamma^0}. This completes the proof.
\end{proof}

\begin{proof}[Proof of Lemma~\ref{lem:weak convergence}] 
Let us fix $x\in\R^{2d}$ and $\gamma\in(0,1/2)$. For $\mathrm{x} \in\mathcal{C}([0,t],\R^{2d})$ we denote 
$$\Vert \mathrm{x}\Vert_\gamma:=\sup_{0\leq r<s\leq t}\frac{\vert \mathrm{x}_s-\mathrm{x}_r\vert}{\vert s-r\vert^\gamma}.$$ 
 
It follows from the linear growth~\eqref{eq:uniform_linear_growth} and the uniform convergence~\eqref{eq:unif conv sigma} in Lemma~\ref{lem:cv F_n sigma_n} that for all $t>0$, $m\geq1$, there exists a constant $C_{t,m}>0$ that can be taken independently of $n$ such that for all $r,s\in[0,t]$,
$$\sup_{n\geq1}\E_x\left[\vert X^{\flat,n}_s-X^{\flat,n}_r\vert^{2m}\right]\leq C_{t,m}\vert s-r\vert^m.$$
In addition, it can be deduced from~\cite[Theorem 10.1, p. 152]{SchillingPartzsch} that
$$\sup_{n\geq1}\E_x\left[\Vert X^{\flat,n}\Vert_\gamma\right]<\infty.$$

We now denote by $\mathrm{P}_x^{\flat,n}$ the law in the space $\mathcal{C}([0,t],\R^{2d})$ of the process $(X^{\flat,n}_s)_{s \in [0,t]}$ under $\mathbb{P}_x$. We shall also denote by $(\mathrm{X}_s)_{s \in [0,t]}$ the canonical process on this space. For $M>0$, let us define the set 
$$\mathcal{C}^{M}_x:=\{\mathrm{x}\in\mathcal{C}([0,t],\R^{2d}):\Vert \mathrm{x}\Vert_\gamma\leq M, \mathrm{x}_0=x\},$$
which is compact by Arzelà-Ascoli theorem. Let $\epsilon>0$, there exists $M_\epsilon>0$ large enough by Markov inequality such that for all $n\geq1$, $$\mathrm{P}^{\flat,n}_x(\Vert \mathrm{X}\Vert_\gamma>M_\epsilon)\leq\frac{\sup_{n\geq1}\mathbb{E}_x\left[\Vert X^{\flat,n}\Vert_\gamma\right]}{M_\epsilon}\leq\epsilon.$$
Therefore, for all $n\geq1$,
$$\mathrm{P}_x^{\flat,n}((\mathcal{C}^{M_\epsilon}_x)^c)\leq\epsilon,$$
which ensure the tightness of the family of measures $(\mathrm{P}^{\flat,n}_x)_{n \geq 1}$. By Prokhorov's theorem, there exists a subsequence, still denoted by
$({\mathrm P}^{\flat,n}_x)_{n \geq 1}$ for convenience, which weakly converges to a probability measure $\mathrm{P}^\flat_x$ on the space $\mathcal{C}([0,t],\R^{2d})$. In particular it follows that $\mathrm{P}^\flat_x(X_0=x)=1$. The objective is now to prove that $\mathrm{P}^\flat_x$ is the law of $(X^\flat_s)_{s \in [0,t]}$ with initial condition $x$.

Let $s>0$ and $0\leq u_1\leq\ldots\leq u_m< s<r\leq t$, $h_j\in\mathcal{C}^b(\R^{2d})$ and $\varphi\in\mathcal{C}^2_c(\R^{2d})$. Let us define $\Phi_n:\mathcal{C}([0,t],\R^{2d})\to\R$ by
$$\Phi_n(\mathrm{x}):=\left[\varphi(\mathrm{x}_r)-\varphi(\mathrm{x}_s)-\int_s^r\mathcal{L}_{F^\flat_n,\sigma^\flat_n}\varphi(\mathrm{x}_v)\mathrm{d}v\right]\prod_{j=1}^mh_j(\mathrm{x}_{u_j}).$$

Similarly, we define $\Phi$ with $\mathcal{L}_{F^\flat,\sigma^\flat}$. Since $\mathrm{P}^{\flat,n}_x$ solves the martingale problem with coefficients $F^\flat_n$, $\sigma^\flat_n$, we know that $\mathrm{E}^{\flat,n}_x\Phi_n=0$ and we wish to prove that $\mathrm{E}^\flat_x\Phi=0$. There are two difficulties: first, the convergence estimates of Lemma~\ref{lem:cv F_n sigma_n} are not sufficient to deduce that $\Phi_n$ converges to $\Phi$; second, since $F^\flat$ may be discontinuous, the map $\mathrm{x}\mapsto\int_s^r F^\flat(\mathrm{x}_v)\cdot\nabla_p\varphi(\mathrm{x}_v)\mathrm{d}v$ may not be continuous on the set $\mathcal{C}([0,t], \R^{2d})$ (it may be discontinuous even on the subset of constant functions). To overcome these difficulties, we rely on the quantitative upper bound~\eqref{eq:chaudru-gaussian-bound} on the transition density of weak (and, a fortiori, strong) solutions to~\eqref{eq:Langevin} in the setting of Proposition~\ref{prop:chaudru}. Combining this bound with Lemma~\ref{lem:cv F_n sigma_n}, we deduce that there exists $C^\flat(t)$ such that for any $n \geq 1$ and $s \in (0,t]$, the density $p^\flat_n(s;x,y)$ of $X^{\flat,n}_s$ under $\mathbb{P}_x$ satisfies
\begin{equation*}
    p^\flat_n(s;x,y) \leq \frac{C^\flat(t)}{s^{2d}}.
\end{equation*}
As a consequence, for all $s\in(0,t]$ and for any  $L^1$ function $h:\R^{2d} \to \R$,
\begin{equation*}
    \mathrm{E}^{\flat,n}_x h(\mathrm{X}_s) =\int_{\R^{2d}} h(y) p^\flat_n(s;x,y)\dd y\leq \frac{C^\flat(t)}{s^{2d}}\|h\|_{L^1}.
\end{equation*}
Hence, for $n\geq1$,
  \begin{align*}
  \mathrm{E}^{\flat,n}_x\left|\Phi-\Phi_n\right|&\leq C\,\mathrm{E}^{\flat,n}_x\left[\int_s^r|F^\flat_n(\mathrm{X}_v)-F^\flat(\mathrm{X}_v)||\nabla_p\varphi(\mathrm{X}_v)|\mathrm{d}v\right]+C||\sigma^\flat-\sigma^\flat_n||_\infty\\
  &\leq C' C^\flat(t)\int_s^t\frac{\mathrm{d}r}{r^{2d}}\,\int_K|F^\flat_n(y)-F^\flat(y)|\mathrm{d}y+C||\sigma^\flat-\sigma^\flat_n||_\infty,
  \end{align*} 
  for some positive constants $C$ and $C'$, where $K$ is a compact set which contains the support of $\nabla_p \varphi$. Therefore, it follows from~\eqref{eq:unif conv sigma} and~\eqref{eq:cv-F_n} that $\mathrm{E}^{\flat,n}_x\left|\Phi-\Phi_n\right|\underset{n\rightarrow\infty}{\longrightarrow}0$. 
  Consequently, for $n\geq1$, 
  $$\vert\mathrm{E}_x^{\flat,n}\Phi_n-\mathrm{E}^\flat_x\Phi\vert\leq\mathrm{E}_x^{\flat,n}\vert\Phi_{n}-\Phi\vert+\vert\mathrm{E}_x^{\flat,n}\Phi-\mathrm{E}^\flat_x\Phi\vert.$$
  We now show that $\mathrm{E}_x^{\flat,n}\Phi \to \mathrm{E}^\flat_x\Phi$. As was already mentioned above, one must take care to the fact that, in the definition of $\Phi$, the term $\int_s^r F^\flat(\mathrm{x}_v) \cdot \nabla_p \varphi(\mathrm{x}_v)\dd v$ may be a discontinuous function of $(\mathrm{x}_s)_{s \in [0,t]}$. So we show that
  \[
    \mathrm{E}^{\flat,n}_x\left[\int_s^r F^\flat(\mathrm{X}_v) \cdot \nabla_p \varphi(\mathrm{X}_v)\dd v \prod_{j=1}^m h_j(\mathrm{X}_{u_j})\right] \to \mathrm{E}^\flat_x\left[\int_s^r F^\flat(\mathrm{X}_v) \cdot \nabla_p \varphi(\mathrm{X}_v)\dd v \prod_{j=1}^m h_j(\mathrm{X}_{u_j})\right],
  \]
  the convergence of the remaining terms in $\Phi$ simply follow from the weak convergence of $\mathrm{P}^{\flat,n}_x$ to $\mathrm{P}^\flat_x$. Let us denote by $I^n$ and $I$ the terms in the left- and right-hand side, respectively. Moreover, for $\delta>0$, let $F^\flat_\delta$ be a continuous regularization of $F$, which is bounded on the support of $\nabla_p \varphi$ uniformly in $\delta$ and converges in $L^1_\mathrm{loc}$ to $F$ when $\delta \to 0$, and let $I^n_\delta$, $I_\delta$ be the same quantities as $I^n$, $I$ but with $F^\flat_\delta$ in place of $F^\flat$. First, since $F^\flat_\delta$ is continuous and bounded uniformly in $\delta$ on the support of $\nabla_p \varphi$, we have $I^n_\delta \to I_\delta$ for any value of $\delta>0$. Second, we have, by the Markov property for $(X^{\flat,n}_s)_{s \in [0,t]}$,
  \begin{align*}
      |I^n_\delta - I^n| &= \left|\mathbb{E}_x\left[\int_s^t \mathbb{E}_{X^{\flat,n}_{u_m}}\left[(F_\delta^\flat(X^{\flat,n}_{v-u_m})-F^\flat(X^{\flat,n}_{v-u_m})) \cdot \nabla_p \varphi(X^{\flat,n}_{v-u_m})\right]\dd v \prod_{j=1}^m h_j(X^{\flat,n}_{u_j})\right]\right|\\
      &\leq C \int_s^t \int_K |F^\flat_\delta(y)-F^\flat(y)| \mathbb{E}_x\left[p^\flat_n(v-u_m;X^{\flat,n}_{u_m},y)\right]\dd y \dd v\\
      &\leq C(t-s)\frac{C^\flat(t-u_m)}{(s-u_m)^{2d}}\int_K |F^\flat_\delta(y)-F^\flat(y)|\dd y,
  \end{align*}
  where $K$ is a compact set which contains the support of $\nabla_p \varphi$ and $C$ is a uniform bound on $\prod_{j=1}^m h_j$. Therefore $|I^n_\delta - I^n|$ converges to $0$ when $\delta \to 0$, uniformly in $n$. The same argument shows that $|I_\delta - I|$ goes to $0$, which finally proves that $I^n \to I$. 

  Since $\mathrm{E}_x^{\flat,n}\Phi_n=0$ for all $n$, we deduce that $\mathrm{E}^\flat_x\Phi=0$. With classical arguments for Dynkin systems one can replace $h_j$ by bounded $\mathcal{G}_s$-measurable functions, where $\mathcal{G}_s=\sigma(\mathrm{X}_u,u\in[0,s])$ since this is valid for all $0<s\leq r\leq t$ and all smooth compactly supported $\varphi$, it follows that $\mathrm{P}^\flat_x$ solves the martingale problem on $(\mathcal C([0,t],\mathbb{R}^{2d}),(\mathcal{G}_s)_{s\in[0,t]})$ with coefficients $F^\flat$, $\sigma^\flat$ and is therefore the distribution of the unique weak solution to~\eqref{eq:Langevin} by Proposition~\ref{prop:chaudru}.
\end{proof}

\begin{proof}[Proof of Lemma~\ref{lem:gamma^0}]
  Lemma~\ref{lem:gamma^0} is an adaptation to the kinetic SDE~\eqref{eq:Langevin} with coefficients $F^\flat$, $\sigma^\flat$ of~\cite[Proposition~2.8~(i)]{LelRamRey}, with $x \in \Gamma^0$, which is written for the Langevin process. The domain $\mathcal{O}^\flat$ in the present article satisfies the same assumptions as the domain $\mathcal{O}$ in~\cite{LelRamRey}. Therefore, the proof closely follows the lines of~\cite[Sect.~3.4]{LelRamRey}. The first step, detailed in~\cite[Sect.~3.4.2]{LelRamRey}, consists in reducing the problem to the study of the process $(\Check{q}_t,\Check{p}_t)_{t\geq0}$ solution to
  \begin{equation*}
   \left\{
     \begin{aligned}
 &        \mathrm{d}\Check{q}_t=\Check{p}_t\mathrm{d}t , \\
  &       \mathrm{d}\Check{p}_t= \sigma^\flat(\Check{q}_t,\Check{p}_t) \mathrm{d}B_t,
     \end{aligned}
  \right.
  \end{equation*} 
  thanks to the Girsanov theorem, whose application is made legitimate by the Gaussian upper bound on the transition density of $(q^\flat_t,p^\flat_t)_{t \geq 0}$ stated in~\cite[Theorem 1.1]{Chaudru2022}.

  Next, following the lines of~\cite[Sect.~3.4.3]{LelRamRey} with $\sigma B_s$ replaced by $M_s = \int_0^s \sigma^\flat(\Check{q}_r,\Check{p}_r)\mathrm{d} B_r$ shows that, after Eq.~(52) there, to conclude the argument one needs to find a positive function $\Psi(t)$ such that, almost surely,
  \begin{equation*}
      \limsup_{t \to 0} \frac{\int_0^t \widehat{M}_s \mathrm{d}s}{\Psi(t)} > 0, \qquad \lim_{t \to 0} \frac{t^2}{\Psi(t)} = 0,
  \end{equation*}
  with $\widehat{M}_s := M_s \cdot u$ where the unit vector $u \in \R^d$ is defined with the notation of~\cite[Sect.~3.4.3]{LelRamRey} by 
  \begin{equation*}
      u = \frac{d_0(\phi)^{-T}(e_d)}{|d_0(\phi)^{-T}(e_d)|}.
  \end{equation*}
  Defining \begin{equation}\label{eq:hatsigma}
  \widehat{\sigma}^2_s = |\sigma^\flat(\Check{q}_s,\Check{p}_s)u|^2 \in [c_1,c_2]
  \end{equation}
  (where we used Assumption~\ref{ass:A2glob} to get the lower and upper bounds), we then deduce from integration by parts and the Dambis-Dubins-Schwarz theorem that there exist two Brownian motions $\beta^1$, $\beta^2$ such that
  \begin{equation*}
      \int_0^t \widehat{M}_s \mathrm{d}s = t \widehat{M}_t - \int_0^t s \mathrm{d}\widehat{M}_s = t \beta^1_{\int_0^t \widehat{\sigma}_s^2 \mathrm{d}s} - \beta^2_{\int_0^t s^2\widehat{\sigma}_s^2 \mathrm{d}s}.
  \end{equation*}
  Let us set, for $t >0$, $\Psi(t) = \sqrt{2 t^3 \log \log(t^{-1})}$ (which indeed satisfies $\lim_{t \to 0} t^2 / \Psi(t)=0$) and write
  \begin{equation*}
      \frac{\int_0^t \widehat{M}_s \mathrm{d}s}{\Psi(t)} = A^1_t \sqrt{\frac{\ell^1_t}{t}\int_0^t \widehat{\sigma}^2_s \mathrm{d}s} - A^2_t\sqrt{\frac{\ell^2_t}{t^3}\int_0^t s^2\widehat{\sigma}^2_s \mathrm{d}s},
  \end{equation*}
  where 
  \begin{equation*}
      A^1_t = \frac{\beta^1_{\int_0^t \widehat{\sigma}_s^2 \mathrm{d}s}}{\sqrt{2 (\int_0^t \widehat{\sigma}_s^2 \mathrm{d}s) \log \log ((\int_0^t \widehat{\sigma}_s^2 \mathrm{d}s)^{-1})}}, \qquad A^2_t = \frac{\beta^2_{\int_0^t s^2\widehat{\sigma}_s^2 \mathrm{d}s}}{\sqrt{2 (\int_0^t s^2\widehat{\sigma}_s^2 \mathrm{d}s) \log \log ((\int_0^t s^2\widehat{\sigma}_s^2 \mathrm{d}s)^{-1})}},
  \end{equation*}
  and
  \begin{equation*}
      \ell^1_t = \frac{\log \log ((\int_0^t \widehat{\sigma}_s^2 \mathrm{d}s)^{-1})}{\log \log (t^{-1})}, \qquad \ell^2_t = \frac{\log \log ((\int_0^t s^2\widehat{\sigma}_s^2 \mathrm{d}s)^{-1})}{\log \log (t^{-1})}.
  \end{equation*}
  The Law of the Iterated Logarithm for the Brownian motions $\beta^1$ and $\beta^2$ implies that
  \begin{equation*}
      \limsup_{t \to 0} A^1_t = \limsup_{t \to 0} A^2_t = 1,
  \end{equation*}
  while elementary estimates using~\eqref{eq:hatsigma} yield
  \begin{equation*}
      \lim_{t \to 0} \ell^1_t = \lim_{t \to 0} \ell^2_t = 1.
  \end{equation*}
  Last, the continuity of $\sigma^\flat$ provided by Assumption~\ref{ass:A2glob} implies that, almost surely,
  \begin{equation*}
      \lim_{t \to 0} \frac{1}{t}\int_0^t \widehat{\sigma}^2_s \mathrm{d}s = \widehat{\sigma}^2_0, \qquad \lim_{t \to 0} \frac{1}{t^3}\int_0^t s^2\widehat{\sigma}^2_s \mathrm{d}s = \frac{\widehat{\sigma}^2_0}{3}.
  \end{equation*}
  Combining the last three results leads to the estimate
  \begin{equation*}
      \limsup_{t \to 0} \frac{\int_0^t \widehat{M}_s \mathrm{d}s}{\Psi(t)} \geq \widehat{\sigma}_0 \left(1 - \frac{1}{\sqrt{3}}\right) \geq \sqrt{c_1}\left(1 - \frac{1}{\sqrt{3}}\right) > 0,
  \end{equation*}
  which completes the proof.
\end{proof}

\subsection{Completion of the proof of Theorem~\ref{thm:Harnack}}\label{ss:unreg-ind}

Combining the results of Propositions~\ref{prop:u_n Harnack} and~\ref{prop:cv-un}, we deduce that for any continuous, bounded and nonnegative function $f$ on $D$, the function $u_f$ defined by~\eqref{eq:def-uf} satisfies the conclusion of Proposition~\ref{prop:u_n Harnack}. To complete the proof of Theorem~\ref{thm:Harnack}, we now have to establish that this conclusion still holds with $f(x) = \ind{x \in A}$ for any measurable subset $A \subset D$. Again, we use a regularization argument and consider, for $\delta >0$
$$\forall x \in D, \qquad f_\delta(x)=\int_{\R^{2d}} \ind{x-y \in A} K_\delta(y) \dd y$$
where $(K_\delta)_{\delta >0}$ is a smooth approximation of the Dirac mass at zero. By the Lebesgue differentiation theorem, $\dd y$-a.e., $\lim_{\delta \to 0} f_\delta(y)=\ind{y \in A}$. By Assumption~\ref{ass:sde} and the Radon--Nikodym theorem, there exists a measurable function $p^D(t;x,y) \leq p(t;x,y)$ such that
\begin{equation*}
    u_{f_\delta}(t,x) = \mathbb{E}_x[\ind{t < \tau_\partial}f_\delta(X_t)] = \int_{\R^{2d}} f_\delta(y) p^D(t;x,y)\dd y,
\end{equation*}
so by the dominated convergence theorem,
\begin{equation*}
    \lim_{\delta \to 0} u_{f_\delta}(t,x) = u_A(t,x),
\end{equation*}
with the notation of Theorem~\ref{thm:Harnack}. This completes the proof of Theorem~\ref{thm:Harnack}.

\subsection{Proof of Corollary~\ref{cor:Harnack}}\label{ss:pf-cor-harnack}

It remains to show the positivity of $u_A$ when $A$ is a measurable subset of $D$ with a non-empty interior. Let us first notice that $u_A(t,x) \ge u_{\mathring{A}}(t,x) \ge 0$. Assume now that there exists $t_0>0$ and $x_0\in D$ such that $u_A(t_0,x_0)=0$, so that $u_{\mathring{A}}(t_0,x_0)=0$. Then by~\eqref{eq:ineq harnack thm} applied to $u_{\mathring{A}}$, for all $t\in(0,t_0)$, $x\in D$, $u_{\mathring{A}}(t,x)=0$. Let us now fix $x_A\in \mathring{A}$. Since $\mathring{A}$ is an open set and the trajectories of $(X_s)_{s\geq0}$ are continuous almost-surely, $u_{\mathring{A}}(t,x_A)\underset{t\rightarrow0}{\longrightarrow}1$ which is in contradiction with the initial statement.

\section{Proof of Theorem~\ref{thm:critere qsd}}\label{sec:critere qsd}

In all this section, we let Assumptions~\ref{ass:A1loc}, \ref{ass:A2loc}, \ref{ass:sde}, \ref{ass:O} and~\ref{ass:lyapunov} hold. The first step of the proof of Theorem~\ref{thm:critere qsd} consists in checking that there exists a compact set $K \subset D$ on which the following conditions are satisfied:
\begin{itemize}
    \item[($F_1$)] (Local Dobrushin condition) There exist $C_1,t_1>0$ and $\Tilde{\nu}$ a probability measure on $K$ such that $$\forall x\in K,\qquad\forall A\in\mathcal{B}(D),\qquad\mathbb{P}_x(t_1<\tau_\partial,X_{t_1}\in A)\geq C_1 \Tilde{\nu}(A\cap K)$$
    where   $\mathcal{B}(D)$ denotes the Borel sets of $D$.
    \item[($F_2$)] (Global Lyapunov condition) There exist  $C_2,t_2>0$, $\psi_1:D\to[1,\infty)$ a measurable function, and $\alpha_1>\alpha_2>0$ such that 
    \begin{enumerate}[label=(\roman*),ref=\roman*]
        \item\label{F2 a} $\forall x\in D,\quad\mathbb{E}_x[\ind{t_2<\tau_K\land\tau_{\partial}}\psi_1(X_{t_2})]\leq\mathrm{e}^{-\alpha_1t_2}\psi_1(x)$, with $\tau_K:=\inf\{t>0:X_t\in K\}$;
        \item\label{F2 c} $\forall x\in K,\,\forall t\in[0,t_2],\quad\mathbb{E}_x[\ind{t<\tau_\partial}\psi_1(X_{t})]\leq C_2$;
        \item\label{F2 b}$\forall x\in K,\quad\mathrm{e}^{\alpha_2t}\mathbb{P}_x(t < \tau_\partial, X_{t}\in K)\underset{t\rightarrow\infty}{\longrightarrow}\infty$. 
    \end{enumerate}
    \item[($F_3$)] (Harnack inequality) There exists $C_3>0$ such that for all $t\geq0$, $$\sup_{x\in K}\mathbb{P}_x(\tau_\partial>t)\leq C_3\inf_{x\in K}\mathbb{P}_x(\tau_\partial>t).$$
\end{itemize}

This step is detailed in Section~\ref{sec:ProofConditionF}. These conditions then allow to apply the abstract result~\cite[Theorem 3.5]{CV}, which leads to the conclusion of Theorem~\ref{thm:critere qsd}. This is done in Section~\ref{sec:ConclusionProof}.

\begin{remark}
    By the Feller property for $(X_t)_{t \geq 0}$ given by Assumption~\ref{ass:sde}, one has that the strong Markov property, also needed as Condition~($F_0$) in~\cite{CV}, is automatically satisfied.
\end{remark}

\subsection{Proof of Conditions~$(F_1)$,~$(F_2)$ and~$(F_3)$}
\label{sec:ProofConditionF}

For $k\geq1$, let $K_k$ be the following compact subset of $D$
\begin{equation}\label{eq:K_k}
    K_k=\left\{(q,p)\in D: \vert q\vert\leq k, \mathrm{d}_\partial(q)\geq\frac{1}{k}, \vert p\vert\leq k\right\}  
\end{equation} 
where we denote by $\mathrm{d}_\partial(q)$ the Euclidean distance from any point $q \in \mathcal O$ to $\partial\mathcal{O}$.

\begin{proposition}\label{qsd langevin CRAS}
There exists $\alpha_{2,\mathrm{min}}>0$ such that for any $\alpha_2 > \alpha_{2,\mathrm{min}}$, for any $\lambda > \alpha_2$ and for any $p \in (1, \lambda/\alpha_2)$, there exists $k_0$ such that for all $k\geq k_0$ the hypotheses ($F_1$), ($F_2$) and ($F_3$) are satisfied with $K=K_{k}$, $\psi_1 = \phi_\lambda^{1/p}$, where $\phi_\lambda : D \to [1,\infty)$ is given by Assumption~\ref{ass:lyapunov}.
\end{proposition}

We first prove that for any $k_0$ such that $K_{k_0}$ has nonempty interior, $(F_1)$ and $(F_3)$ hold with $K=K_{k_0}$ and $C_1$, $t_1$, $\tilde{\nu}$ and $C_3$ which depend on $k_0$. 

\begin{proof}[Proof of $(F_1)$ and $(F_3)$] Let $k_0$ be large enough such that $K_{k_0}$ has a non-empty interior, and $u_{K_{k_0}}$ be the function defined in \eqref{eq:def u_A}. Let $x_0\in K_{k_0}$, $t_0>0$. Theorem~\ref{thm:Harnack} ensures the existence of $c>0$ such that for all $x\in K_{k_0}$, and all $A$ measurable subset of $D$,
\begin{align*}
    \mathbb{P}_x(X_{t_0+1}\in A,\tau_\partial>t_0+1)&\geq c\, \mathbb{P}_{x_0}(X_{t_0}\in A,\tau_\partial>t_0)\\
    &\geq c \underbrace{\mathbb{P}_{x_0}(X_{t_0}\in K_{k_0},\tau_\partial>t_0)}_{=u_{K_{k_0}}(t_0,x_0)>0\text{ by Corollary~\ref{cor:Harnack}}}\underbrace{\frac{\mathbb{P}_{x_0}(X_{t_0}\in A\cap K_{k_0},\tau_\partial>t_0)}{\mathbb{P}_{x_0}(X_{t_0}\in K_{k_0},\tau_\partial>t_0)}}_{=:\Tilde{\nu}(A\cap K_{k_0})},
\end{align*}
where $\Tilde{\nu}$ is a probability measure on $K_{k_0}$, which yields $(F_1)$.

Theorem~\ref{thm:Harnack} applied to $u_D$ now ensures the existence of $c>0$ such that for all $t\geq1$,
$$\sup_{x\in K_{k_0}}\mathbb{P}_x(\tau_\partial>t)\leq c\inf_{x\in K_{k_0}}\mathbb{P}_x(\tau_\partial>t+1)\leq c\inf_{x\in K_{k_0}}\mathbb{P}_x(\tau_\partial>t),$$ since $u_D$ is a nonincreasing function of $t$. It remains to prove such an inequality for $t\in[0,1]$. For all $t\in[0,1]$, $x,y\in K_{k_0}$,
$$\mathbb{P}_x(\tau_\partial>t)\leq1\leq \frac{\mathbb{P}_y(\tau_\partial>t)}{\mathbb{P}_y(\tau_\partial>1)}.$$
Besides, the function $y\in K_{k_0}\mapsto\mathbb{P}_y(\tau_\partial>1)$ admits a positive lower-bound by Theorem~\ref{thm:Harnack}. Hence, for all $x,y\in K_{k_0}$ and $t\in[0,1]$,
$$\mathbb{P}_x(\tau_\partial>t)\leq c'\mathbb{P}_y(\tau_\partial>t),$$
for some $c'>0$. Therefore, for all $t\geq0$,
$$\sup_{x\in K_{k_0}}\mathbb{P}_x(\tau_\partial>t)\leq (c\lor c')\inf_{x\in K_{k_0}}\mathbb{P}_x(\tau_\partial>t),$$ which concludes the proof of $(F_3)$.
\end{proof}

We next show that for any compact subset $K$, for any $t_2>0$, $\lambda>0$ and $p>1$, ($F_2$)-\eqref{F2 c} holds true with $C_2$ which depends on $K$, $t_2$, $p$ and $\lambda$, and $\psi_1 = \phi_\lambda^{1/p}$. 

\begin{proof}[Proof of ($F_2$)-\eqref{F2 c}] 
    We prove that, for all $\lambda>0$, for all $x\in D$ and all $t\geq 0$,
    \begin{equation}\label{eq:lyapunov_on_D}
    \mathbb{E}_x[\ind{t<\tau_\partial}\phi_\lambda(X_{t})]\leq \mathrm{e}^{c_\lambda t}\phi_\lambda(x).
    \end{equation}
    By Jensen's inequality, this implies that for any choice of $t_2>0$, $\lambda>0$, $p>1$ and $K$ compact in $D$, ($F_2$)-\eqref{F2 c} holds with $\psi_1=\phi_\lambda^{1/p}$ since $\phi_\lambda$ is $\mathcal C^2$ on $D$ and thus bounded on any compact set $K$.
    
    Indeed, let $\lambda>0$ and $x\in D$. Let $\tau_{K_k^c}=\inf\{t>0: X_t \not \in K_k\}$.
    Let $k$ be large enough such that $x \in K_k$. Applying Itô's formula to the process $(\mathrm{e}^{-c_\lambda t}\phi_\lambda(X_{t}))_{t\geq0}$ up to time $t\wedge \tau_{K_k^c}$ one has, $\mathbb{P}_x$-almost surely, 
    \[\begin{aligned}
    \mathrm{e}^{-c_\lambda(t\land\tau_{K_k^c})}\phi_\lambda(X_{t\land\tau_{K_k^c}})&=\phi_\lambda(x)+\int_0^t\ind{s<\tau_{K_k^c}}\mathrm{e}^{-c_\lambda s}(\mathcal{L}_{F,\sigma}\phi_\lambda(X_s)-c_\lambda\phi_\lambda(X_s))\mathrm{d}s\\ 
    &\quad + \int_0^t\ind{s<\tau_{K_k^c}}\mathrm{e}^{-c_\lambda s}\nabla_p\phi_\lambda(X_s)\cdot\sigma(X_s)\mathrm{d}B_s.
    \end{aligned}\]
    Since 
    $\nabla_p \phi_\lambda$ is continuous and therefore bounded on $K_k$, one obtains:
    \begin{align*}
    \mathbb E_x\left[\mathrm{e}^{-c_\lambda(t\land\tau_{K^c_k})}\phi_\lambda(X_{t\land\tau_{K^c_k}})\right]
    &\leq \phi_\lambda(x)+\mathbb E_x\left[\int_0^t\ind{s<\tau_{K^c_k}}\mathrm{e}^{-c_\lambda s}(\mathcal{L}_{F,\sigma}\phi_\lambda(X_s)-c_\lambda\phi_\lambda(X_s)) \, \dd s\right]\\
    &\leq \phi_\lambda(x),
    \end{align*}
    because $\mathcal{L}_{F,\sigma}\phi_\lambda-c_\lambda \phi_\lambda\leq0$ on $D$ since $\phi_\lambda\geq 1$. Since $\phi_\lambda$ is nonnegative, one gets
    $$  \mathbb{E}_x\left[\ind{t<\tau_{K ^c_k}}\phi_\lambda(X_{t})\right]\leq \mathrm{e}^{c_\lambda t}\phi_\lambda(x),
  $$
    which proves~\eqref{eq:lyapunov_on_D} by sending $k$ to $\infty$, using the fact that almost surely, $\mathbb{1}_{\tau_{K ^c_k}>t}$ converges to $\mathbb{1}_{\tau_{\partial}>t}$ thanks to the continuity of the process $(X_t)_{t\ge 0}$ and the fact that $D=\cup_{k \ge 0} K_k$.
\end{proof}

We now show that for any $k_1$ such that $K_{k_1}$ has nonempty interior, there is an $\alpha_{2,\mathrm{min}}(k_1)>0$ such that ($F_2$)-\eqref{F2 b} holds true with $K=K_{k_0}$ for any $k_0\geq k_1$ and any $\alpha_2 > \alpha_{2,\mathrm{min}}(k_1)$.

\begin{proof}[Proof of ($F_2$)-\eqref{F2 b}] 

Let $k_1\geq1$ large enough such that $L:=K_{k_1}$ has a non-empty interior. Let $B$ be a nonempty open ball such that $\overline{B}$ is included in the interior of $L$. This ensures that $\delta:=\mathrm{d}(L^c,B)>0$. Let us first prove that there exists $t_0>0$ small enough such that
\begin{equation}\label{eq:inf proba ball B and small time}
    c_1:=\inf_{y\in \overline{B},s\in[0,t_0]}\mathbb{P}_y\left(X_s\in L, \tau_\partial>s\right)>0.
\end{equation}

Let $t_0>0$ and let $y\in\overline{B}$, $s\in[0,t_0]$. One has that 
\begin{align*}
\mathbb{P}_y\left(X_s\in L, \tau_\partial>s\right)
& \ge 1 - \mathbb{P}_y\left(X_s\not \in L\right) - \mathbb{P}_y \left( \tau_\partial\le s\right)\\
& \ge 1 - \mathbb{P}_y\left(X_s\not \in L\right) - \mathbb{P}_y \left( \tau_{L^c}\le s\right)\\
& \ge 1 - 2 \, \mathbb{P}_y \left( \tau_{L^c}\le s\right)\\
& \ge 1 - 2 \, \mathbb{P}_y \left( \sup_{r \in [0,s]} \mathrm{d}(X_r,\overline{B}) \geq \delta \right).
\end{align*}
As a consequence, 
\begin{equation*}
\inf_{s \in [0,t_0],y\in\overline{B}} \mathbb{P}_y\left(X_s\in L, \tau_\partial>s\right)\ge 1 - 2 \, \sup_{y\in\overline{B}}\mathbb{P}_y \left( \sup_{r \in [0,t_0]} \mathrm{d}(X_r,\overline{B}) \geq \delta \right).
\end{equation*}
By~\eqref{eq:loc-estim:2}, one may take $t_0$ small enough so that
$$\sup_{y\in\overline{B}}\mathbb{P}_y \left( \sup_{r \in [0,t_0]} \mathrm{d}(X_r,\overline{B}) \geq \delta \right)\leq\frac{1}{4},$$
which implies that~\eqref{eq:inf proba ball B and small time} holds with $c_1\geq 1/2$. 

Now, let us take any $t\geq1$. We can write $t=kt_0+s$ with $k=\lfloor t/t_0\rfloor$ and $s\in[0,t_0)$. As a result, by the Markov property at $kt_0$,
\[\begin{aligned}
\mathbb{P}_x\left(X_{t}\in L,\tau_\partial>t\right)&\geq\mathbb{P}_x\left(X_{kt_0}\in B,X_{t}\in L,\tau_\partial>t\right)\\
&= \mathbb{E}_x\left[\ind{\tau_\partial>kt_0,X_{kt_0}\in B}\, \mathbb{P}_{X_{kt_0}}\left(X_{s}\in L, \tau_\partial>s\right)\right]\\
&\geq c_1\mathbb{P}_x\left(X_{kt_0}\in B,\tau_\partial>kt_0\right)
\end{aligned}\]
where $c_1>0$ by~\eqref{eq:inf proba ball B and small time}.  
Now let $c_2:=\inf_{x\in L}u_B(t_0,x)>0$ by Corollary~\ref{cor:Harnack}. It follows from the Markov property at times $(k-1)t_0,(k-2)t_0,\ldots,t_0$ that for any $x \in L$,
\[\begin{aligned}
\mathbb{P}_x\left(X_{t}\in L,\tau_\partial>t\right)&\geq c_1c_2^{\lfloor t/t_0\rfloor}.
\end{aligned}\]
For any $k_0\geq k_1$, we have $c_3:=\inf_{x\in K_{k_0}} u_B(t_0,x)>0$ by Corollary~\ref{cor:Harnack}. Hence, using the Markov property at time $t_0$, we have, for any $x\in K_{k_0}$ and $t\geq 0$,
\[\begin{aligned}
\mathbb{P}_x\left(X_{t+t_0}\in K_{k_0},\tau_\partial>t+t_0\right)\geq \mathbb{P}_x\left(X_{t+t_0}\in L,\tau_\partial>t+t_0\right)&\geq c_3c_1c_2^{\lfloor t/t_0\rfloor}.
\end{aligned}\]
Taking $\alpha_2>\vert \log(c_2)\vert/t_0 =: \alpha_{2,\mathrm{min}}(k_1)$, we have $\inf_{x\in K_{k_0}}\mathrm{e}^{\alpha_2t}\mathbb{P}_x(X_{t}\in K_{k_0},\tau_\partial>t)\underset{t\rightarrow\infty}{\longrightarrow}\infty$,
which yields \eqref{F2 b} with $K=K_{k_0}$ for any $k_0\geq k_1$  and $\alpha_2 > \alpha_{2,\mathrm{min}}(k_1)$.
\end{proof}

To complete the proof, we show that for $k_1$ and $\alpha_2 > \alpha_{2,\mathrm{min}}(k_1)$ given by~($F_2$)-\eqref{F2 b}, for any $\lambda>\alpha_2$, $\alpha_1 \in (\alpha_2,\lambda)$, $p \in (1, \lambda/\alpha_2)$, one can find $t_2>0$ and $k_0 \geq k_1$ large enough so that ($F_2$)-\eqref{F2 a} holds with $\psi_1=\phi_\lambda^{1/p}$ and $K=K_{k_0}$. This fixes the compact set $K=K_{k_0}$ as well as the values of all parameters in Conditions~$(F_1)$,~$(F_2)$ and~$(F_3)$.

\begin{proof}[Proof of ($F_2$)-\eqref{F2 a}]
    Recall that $D_\lambda$ is defined in Assumption~\ref{ass:lyapunov}. Using the same method as in the proof of~\eqref{eq:lyapunov_on_D}, we obtain that, for  any $\lambda>0$, and for all $t\geq0$ and $x\in D\setminus D_\lambda$, 
    \begin{equation}
    \mathbb{E}_x\left[\ind{t<\tau_{\partial}\land\tau_{D_\lambda}}\phi_\lambda(X_{t})\right]\leq \mathrm{e}^{-\lambda t}\phi_\lambda(x),
    \label{Lyapunov_2}
    \end{equation} 
    where  $\tau_{D_\lambda}=\inf \{t>0: X_t\in D_\lambda\}$. In addition, we have
    \begin{equation}\label{sup proba}
    \sup_{x\in D_\lambda}\mathbb{P}_x(\tau_{\partial}\land\tau_{K_k}>1)\underset{k\rightarrow\infty}{\longrightarrow}0.
    \end{equation} 
    Indeed, recall from Assumption~\ref{ass:sde} that the process $(X_t)_{t\geq 0}$ is constructed for all time $t\geq 0$. Then, for $x\in D_\lambda$,
        $$\mathbb{P}_x(\tau_{\partial}\land\tau_{K_k}>1)\leq\mathbb{P}_x(X_1\in D\cap (K_k)^c).$$
        Applying the Harnack inequality~\eqref{eq:ineq harnack thm} on the compact set $\overline{D}_\lambda\subset\tilde{\mathcal{O}} \times \R^d$ to the probability in the right-hand side above, one has that there exists $C>0$ such that for any fixed $x_0\in \overline{D}_\lambda$,
        $$\sup_{x\in \overline{D}_\lambda}\mathbb{P}_x(X_1\in D\cap (K_k)^c)\leq C\mathbb{P}_{x_0}(X_2\in D\cap (K_k)^c)\underset{k\rightarrow\infty}{\longrightarrow}0,$$
        since $\ind{X_2\in D\cap (K_k)^c}\underset{k\rightarrow\infty}{\longrightarrow}0$ almost-surely. This proves the claim~\eqref{sup proba}.

        Following the lines of~\cite[Section~12.3]{CV} starting from Equation~(12.10) therein, and using~\eqref{eq:lyapunov_on_D},~\eqref{Lyapunov_2}, and~\eqref{sup proba}, one finally shows that for $k_1$ and $\alpha_2 > \alpha_{2,\mathrm{min}}(k_1)$ given by~($F_2$)-\eqref{F2 b}, for any $\lambda>\alpha_2$, $\alpha_1 \in (\alpha_2,\lambda)$, $p \in (1, \lambda/\alpha_2)$, one can find $t_2>0$ and $k_0 \geq k_1$ large enough so that ($F_2$)-\eqref{F2 a} holds with $\psi_1=\phi_\lambda^{1/p}$ and $K=K_{k_0}$.
\end{proof}

\subsection{Conclusion of the proof of Theorem~\ref{thm:critere qsd}}
\label{sec:ConclusionProof} 

For $\alpha_2$ satisfying the assumptions of Proposition~\ref{qsd langevin CRAS}, \cite[Theorem 3.5]{CV} shows that for any $\lambda>\alpha_2$ and $p \in (1, \lambda/\alpha_2)$, there exists $k_0$ such that, for all $k\geq k_0$, there exists a unique QSD $\mu$ which satisfies $\mu(\phi_\lambda^{1/p})<\infty$ and $\mu(K_k)>0$. Moreover, by~\cite[Theorem 3.5]{CV} and~\cite[Corollary~2.7]{CV}, there exist $\lambda_0 \in [0,\alpha_2]$, $C,\beta >0$ and $\varphi:D\to\mathbb{R}^*_+$ such that, for all $x\in D$ and all measurable $f:D\to\mathbb R$ with $|f|\leq \phi_\lambda^{1/p}$,
\begin{equation}
\label{eq:in-proof-thm-specgap}
\left|\mathrm e^{\lambda_0 t}\mathbb E_x\left[\ind{t<\tau_\partial}f(X_t)\right]-\varphi(x)\mu(f)\right|\leq C\mathrm e^{-\beta t}\phi_\lambda^{1/p}(x).
\end{equation}
Since~\eqref{eq:in-proof-thm-specgap} entails that for all bounded and measurable $f$, $\varphi(x)\mu(f)$ is the limit of $\mathrm e^{\lambda_0 t}\mathbb E_x(f(X_t)\mathbb 1_{\tau_\partial>t})$, neither $\mu$ nor $\varphi$ depend on the choice of $\lambda$, $p$ and $k$.
Notice that with the notation introduced before Theorem~\ref{thm:critere qsd}, the measure $\mu$ belongs to the set $\mathcal{M}_{\alpha_2}$ and therefore to the set $\mathcal{M}_{\lambda_0}$ since $\lambda_0 \leq \alpha_2$.

We now show that $\mu$ satisfies $\mu(\phi_\lambda^{1/p}) < \infty$ for any $\lambda>\lambda_0$ and $p \in (1, \lambda/\lambda_0)$.
 Let $\alpha'_2 > \lambda_0$. Taking $f(x) = \ind{x \in K_k}$ in~\eqref{eq:in-proof-thm-specgap} with $k$ large enough such that $(F_2)$-\eqref{F2 b} holds true with $K=K_k$ and using the fact that $\mu(K)>0$, we deduce that $(F_2)$-\eqref{F2 b} actually holds true for $\alpha'_2$ instead of $\alpha_2$. The proof of Proposition~\ref{qsd langevin CRAS} then shows that for any $\lambda > \alpha'_2$ and for any $p \in (1, \lambda/\alpha'_2)$, the hypotheses ($F_1$), ($F_2$) and ($F_3$) are satisfied with $\psi_1 = \phi_\lambda^{1/p}$, which by~\cite[Theorem 3.5]{CV} and~\cite[Corollary~2.7]{CV} again shows that for any $\lambda>\lambda_0$ and $p \in (1, \lambda/\lambda_0)$, there exists a unique QSD $\mu'$ which satisfies $\mu'(\phi_\lambda^{1/p})<\infty$ and $\mu'(K_{k_0})>0$ for $k_0$ large enough; and moreover $\mu'$ satisfies~\eqref{eq:in-proof-thm-specgap}. As above, we deduce that $\mu'=\mu$. 

To complete the proof of Theorem~\ref{thm:critere qsd}, we show the uniqueness of a QSD $\mu$ in $\mathcal{M}_{\lambda_0}$. Let $\tilde{\mu} \in \mathcal{M}_{\lambda_0}$ be a QSD for $(X_t)_{t \geq 0}$. By definition of $\mathcal{M}_{\lambda_0}$, there exist $\tilde{\lambda}>\lambda_0$ and $\tilde{p} \in (1, \tilde{\lambda}/\lambda_0)$ such that $\tilde{\mu}(\phi_{\tilde{\lambda}}^{1/\tilde{p}}) < \infty$. Moreover, since $\tilde{\mu}$ is a QSD for $(X_t)_{t \geq 0}$, for any $t>0$,
\begin{equation*}
    \tilde{\mu}(K) = \frac{\mathbb P_{\tilde{\mu}}(X_t \in K, \tau_\partial > t)}{\mathbb P_{\tilde{\mu}}(\tau_\partial > t)}
\end{equation*} and since $K$ has nonempty interior, by Theorem~\ref{thm:Harnack} the right-hand side is positive. As a consequence, by the uniqueness result given by~\cite[Theorem 3.5]{CV} for $\lambda=\tilde{\lambda}$ and $p=\tilde{p}$, and the fact that $\mu$ does not depend on the choice of $(\lambda,p)$, we conclude that $\tilde{\mu}=\mu$.

\section{Proof of Theorems~\ref{thm:ex-bounded} and~\ref{thm:ex-langevin}}\label{sec:example lyapunov}

\subsection{Proof of Theorem~\ref{thm:ex-bounded}}\label{ss:ex-bounded}

To prove Theorem~\ref{thm:ex-bounded}, we need to construct, under Assumptions~\ref{ass:A1loc}, \ref{ass:A2loc} and~\ref{ass:sde}, and with a \emph{bounded} domain $\mathcal{O}$ satisfying Assumption~\ref{ass:O}, a bounded Lyapunov function $\phi_\lambda$ which satisfies Assumption~\ref{ass:lyapunov}. 

To proceed, we let $g : \mathbb{R}_+^* \to (0,1]$ be a $\mathcal{C}^2$ function such that $g(\rho) =\rho$ for $\rho\in(0,1/2)$ and $g(\rho) =1$ for $\rho \geq 1$. Now let $\beta:=1+\sup_{q\in\mathcal{O}}\vert q\vert$ and $p_a>0$ be a parameter which will be defined later. We define $\phi$ on $D$ by\begin{equation*} 
  \phi(q,p):=\left\{\begin{aligned}
    &\beta-\frac{q\cdot p}{\vert p\vert}g(\vert p\vert),&&\vert p\vert\neq0,\\
    &\beta,&&p=0,
  \end{aligned}\right.
\end{equation*} 
then $\phi\geq 1$ on $D$ and $\phi\in \mathcal{C}^2(D)$. Now, for any $\lambda>0$, let
$p_0(\lambda):=1+p_a+4\lambda\beta$ and $D_\lambda\subset D$ be defined as follows:
$$ D_\lambda= \{(q,p)\in D: \vert p\vert\leq p_0(\lambda)\}.$$ 
Then, for $(q,p)\in D\setminus D_\lambda$, $\phi(q,p)=\beta-\frac{q\cdot p}{\vert p\vert}$ and satisfies
\[\begin{aligned}
\mathcal{L}_{F,\sigma}\phi(q,p)&=p\cdot\nabla_q \phi(q,p)+F(q,p)\cdot\nabla_p\phi(q,p)+ \frac{\sigma\sigma^T(q,p)}{2}:\nabla^2_p\phi(q,p)\\ 
&=-\vert p\vert\left(1+ \frac{q\cdot F(q,p)}{\vert p\vert^2}- \frac{(q\cdot p)(p\cdot F(q,p))}{\vert p\vert^4}
-2\frac{(\sigma^T(q,p)q)\cdot (\sigma^T(q,p)p)}{\vert p\vert^4}\right)\\
&\quad -\vert p\vert\left(3\frac{q\cdot p}{\vert p\vert^6}\vert\sigma^T(q,p)p\vert^2-\frac{p\cdot q}{\vert p\vert^4}\mathrm{Tr}(\sigma\sigma^T(q,p))\right). 
\end{aligned}\]
Since $\mathcal{O}$ is bounded, Assumptions~\ref{ass:A1loc} and~\ref{ass:A2loc} yield 
$$\mathcal{L}_{F,\sigma}\phi(q,p)=-\vert p\vert\left(1+\underset{\vert p\vert\rightarrow\infty}{o(1)}\right).$$ 
As a result, if one defines $p_a$ such that for $\vert p\vert\geq p_a$, and $q\in\mathcal{O}$,
$$\mathcal{L}_{F,\sigma}\phi(q,p)\leq-\frac{\vert p\vert}{2},$$ 
then one gets, for $\vert p\vert\geq p_0(\lambda)$,
\[\begin{aligned}
\mathcal{L}_{F,\sigma}\phi(q,p)&\leq-\frac{\vert p\vert}{2}\leq -2\lambda\beta\leq -\lambda\left(\beta-\frac{q\cdot p}{\vert p\vert}\right)=-\lambda\phi(q,p). 
\end{aligned}\] 
Now, on the bounded set $D_\lambda$, $\mathcal{L}_{F,\sigma}\phi(q,p)+\lambda\phi(q,p)$ is bounded by some $c_\lambda<\infty$. Combined with the inequality above this yields that $\mathcal{L}_{F,\sigma}\phi\leq-\lambda\phi+c_\lambda\mathbb{1}_{D_\lambda}$ on $D$. Furthermore, by construction of $\beta$ and $g$, $1\leq\phi\leq 2\beta-1$, so Assumption~\ref{ass:lyapunov} is satisfied and $\phi$ is bounded.

\subsection{Proof of Theorem~\ref{thm:ex-langevin}}\label{ss:ex-langevin}

In this section, we consider the Langevin process~\eqref{eq:Langevin_intro}, with mass matrix $M=I_d$ and an interaction force $F$ given by~\eqref{eq:ex-langevin:1}--\eqref{eq:ex-langevin:2}, with a vector field $\ell$ which is bounded on bounded sets of $\R^d$. We assume that Assumption~\ref{ass:sde} holds true, and fix a set $\mathcal{O} \subset \R^d$ which satisfies Assumption~\ref{ass:O}. Hence, as in Section~\ref{ss:ex-bounded}, the proof of Theorem~\ref{thm:ex-langevin} reduces to the verification of Assumption~\ref{ass:lyapunov}. This is the object of the next statement, which relies on the Lyapunov function introduced in~\cite[Section 4]{LRS}. In the proof, we shall simply denote by $\mathcal{L}$, rather than $\mathcal{L}_{F,\sigma}$, the infinitesimal generator of the associated dynamics.

\begin{lemma}[Lyapunov function for Theorem~\ref{thm:ex-langevin}]\label{lyapunov lemma unbounded}
    In the setting described above, Assumption~\ref{ass:lyapunov} is satisfied.
\end{lemma}

\begin{proof}
Let us define the following function $H$ as in~\cite[Definition 4.1]{LRS}, for all $(q,p)\in\R^{2d}$,
$$H(q,p):=U(q) + \frac{1}{2}|p|^{2}+\frac{\gamma-\delta}{2} q \cdot p +\frac{(\gamma-\delta)^{2}}{4}|q|^{2},$$
where $\delta\in(0,\gamma)$ is a constant small enough such that, cf.~\cite[Equation 4.2]{LRS},  
$$\frac{\delta(\gamma-\delta)}{2}\leq\alpha\;,\qquad\frac{2\delta}{\gamma-\delta}\leq\alpha\;,\;\;\;\;\text{and}\qquad\beta^{2}\leq\gamma(\gamma-\delta),$$
where $\alpha,\beta$ are the constants appearing in~\eqref{eq:ex-langevin:2}. A computation following~\cite[Lemma 4.3]{LRS} yields that for all $(q,p)\in\mathbb{R}^{2d}$,
$$\mathcal{L}H(q,p)\leq-\delta H(q,p)+d\epsilon,$$
where, for notational consistency between~\eqref{eq:Langevin_intro} and~\cite{LRS}, we have written $\epsilon=\gamma \mathrm{k} T$. Let us now define the function $\widehat{H}$ for $(q,p)\in\mathbb{R}^{2d}$,
$$\widehat{H}(q,p)=1+H(q,p).$$

Then, following the computation performed in~\cite[Lemma 4.3]{LRS} one has that for all $n\geq1$ and $(q,p)\in\R^{2d}$,
\begin{equation}\label{eq:lyapunov control}
    \mathcal{L}\widehat{H}^n(q,p)\leq n\widehat{H}^{n-1}(q,p)(-\delta H(q,p))+\epsilon\left[dn\widehat{H}^{n-1}(q,p)+n(n-1)\widehat{H}^{n-2}(q,p)\left|p+\frac{\gamma-\delta}{2}q\right|^2\right].
\end{equation}
Furthermore, using the lower-bound of $H$ in~\cite[Lemma 4.2]{LRS}, one has the existence of a constant $c>0$ such that $|p+\frac{\gamma-\delta}{2}q|^2\leq c\widehat{H}(q,p)$. Moreover, the same lower-bound also ensures that $\widehat{H}(q,p)\underset{|(q,p)|\rightarrow\infty}{\longrightarrow}~\infty$. Consequently, there exists a compact set $B_n\subset\R^{2d}$ such that for all $(q,p)\notin B_n$,
$$\frac{\delta}{2}\widehat{H}(q,p)\geq\delta+d\epsilon+c\epsilon(n-1).$$ 
Therefore, it follows from~\eqref{eq:lyapunov control} that for all $(q,p)\notin B_n$,
\begin{align*}
\mathcal{L}\widehat{H}^n(q,p)&\leq n\widehat{H}^{n-1}(q,p)\left[-\delta \widehat{H}(q,p)+\delta+d\epsilon+c\epsilon(n-1)\right]\\
&\leq-\frac{n\delta}{2}\widehat{H}^{n}(q,p).
\end{align*}

Since $\ell$ is bounded on bounded sets, while $\nabla U$ is continuous, $F(q)-\gamma p$ is bounded on the compact set $B_n$ and so is $\mathcal{L}\widehat{H}^n(q,p)+\frac{n\delta}{2}\widehat{H}^n(q,p)$. Consequently, it is easy to deduce that for all $\lambda>0$ there exists an integer $n_\lambda\geq1$ and $c_{n_\lambda} \geq 0$ such that for all $(q,p)\in\R^{2d}$,
$$\mathcal{L}\widehat{H}^{n_\lambda}(q,p)\leq-\lambda\widehat{H}^{n_\lambda}(q,p)+c_{n_\lambda}\mathbb{1}_{B_{n_\lambda}}(q,p).$$
Moreover, since $H\geq0$ by~\cite[Lemma 4.2]{LRS} then $\phi_\lambda=\widehat{H}^{n_\lambda}\geq1$ satisfies the hypotheses required.
\end{proof}

\subsection*{Acknowledgements} 
The work of N.C. is partially funded by the European Union (ERC, SINGER, 101054787). Views and opinions expressed are however those of the author(s) only and do not necessarily reflect those of the European Union or the European Research Council. Neither the European Union nor the granting authority can be held responsible for them.

The work of T.L. is partially funded by the European Research Council (ERC) under the European Union’s Horizon 2020 research and innovation
programme (project EMC2, grant agreement No 810367), and also from the Agence Nationale de la Recherche through the grant ANR-19-CE40-0010-01 (QuAMProcs). 

The work of J.R. is partially supported by the Agence Nationale de la Recherche through the grants  ANR-19-CE40-0010-01 (QuAMProcs) and ANR-23-CE40-0003 (Conviviality).

\bibliographystyle{plain}
\bibliography{biblio}
\end{document}